\documentclass[twoside]{article}
\usepackage{graphics}
\usepackage{amsmath,amsopn,amssymb,epsfig,psfrag}
\usepackage{color,xcolor}
\usepackage{url}
\usepackage{graphicx}
\usepackage{subfig}
\usepackage{float}
\usepackage{amsfonts}
\usepackage{multirow}
\usepackage{bm}
\usepackage{makecell}
\usepackage{xr}
\usepackage[version=4]{mhchem}

\usepackage[]{mathtools}

\DeclareMathOperator{\bigO}{\mathcal{O}}
\usepackage[]{tikz-cd}
\usepackage[]{dsfont}

\usepackage[colorlinks]{hyperref}

\usepackage[noend]{algpseudocode}
\usepackage{algorithmicx,algorithm}

\usepackage{amsthm}

\newtheorem{theorem}{Theorem}[section]
\newtheorem{lemma}{Lemma}[section]

\theoremstyle{definition}

\newtheorem{remark}{Remark}[section]
\newtheorem{example}{Example}[section]

\DeclareMathOperator{\rank}{rank}
\DeclareMathOperator{\diag}{diag}

\DeclareMathOperator{\T}{\sf T}

\makeatletter
\tagsleft@false\let\veqno\@@eqno
\makeatother

\title{Highly accurate decoupled doubling algorithm for \\ large-scale M-matrix algebraic Riccati equations}

\date{}

\author{
Zhen-Chen Guo\thanks{Department of Mathematics, Nanjing University,  Nanjing  210093,  China; {\tt guozhenchen@nju.edu.cn}} \and 
Eric King-wah Chu\thanks{School of Mathematical Sciences, Monash University, 9 Rainforest Walk, Melbourne, Victoria 3800, Australia; {\tt eric.chu@monash.edu}} \and
Xin Liang\thanks{Yau Mathematical Sciences Center, Tsinghua University, Beijing 10084, China; {\tt liangxinslm@tsinghua.edu.cn} 
}
}
\begin{document}
%
%
%
\newcommand{\bs}{\begin{slide}{}}
\newcommand{\es}{\end{slide}} \newcommand{\be}{\begin{eqnarray*}}
\newcommand{\ee}{\end{eqnarray*}}
\newcommand{\bt}{\begin{tabular}}
\newcommand{\et}{\end{tabular}} \newcounter{bean}
\newcommand{\bl}[1]{\begin{list}{#1}{\usecounter{bean}}} \newcommand{\el}{\end{list}}
\newcommand{\bel}[1]{\begin{equation} \label{#1}} \newcommand{\eel}{\end{equation}}
\newenvironment{mat}{\left[\begin{array}}{\end{array}\right]}
\newenvironment{dt}{\left|\begin{array}}{\end{array}\right|} \newcommand{\bip}{\langle}
\newcommand{\eip}{\rangle} \newenvironment{arr}{\begin{array}}{\end{array}} \def\xa{x_{1}}
\def\xb{x_{2}} \def\xc{x_{3}} \newcommand{\ebs}{\end{slide}\begin{slide}{}}
\newcommand{\bc}{\begin{center}} \newcommand{\ec}{\end{center}}
\newcommand{\Lra}{\Leftrightarrow} \newcommand{\lra}{\leftrightarrow}
\newcommand{\Llra}{\Longleftrightarrow} \newcommand{\La}{\Leftarrow}
\newcommand{\Ra}{\Rightarrow} \newcommand{\la}{\leftarrow} \newcommand{\ra}{\rightarrow}
\def\eop{\hfill $\Box$ \vspace{5mm}}
\def\diag{\mathrm{diag}} \def\rank{\mathrm{rank}}
\newtheorem{lem}{Lemma}[section]
\newtheorem{cor}[lem]{Corollary}
\def\sr{{\cal R}} \def\cs{{\cal C}} \def\vx{{\bf x}}
%
\maketitle

\begin{abstract}
	We consider the numerical solution of large-scale M-matrix  algebraic Riccati equations (MAREs) with low-rank structures. We derive a new doubling iteration, decoupling the original four iteration formulae in the alternating-directional doubling algorithm. 
We prove that the kernels in the decoupled recursion  
are small M-matrices. Illumined by the highly accurate algorithm proposed in J.-G. Xue and R.-C. Li (2017)~\cite{xueL2017highly}, we construct 
the novel triplet representations for the small M-matrix kernels. 
And with these triplet representations, we develop a highly  accurate doubling algorithm, named dADDA, 
for large-scale MAREs with low-rank structures, where the GTH-like algorithm is applied for solving the associated linear systems.  Benefiting  from the decoupled form, the proposed dADDA   
can utilize the special structures that may exist in the original $A$ and $D$. For example, with  $A$ and $D$ being banded sparse or low-rank  updates with diagonal matrices,   the dADDA  only requires 
$\bigO(n+m)$ flops in each iteration. 
Illustrative numerical examples will be presented on the efficiency of our algorithm. 
\end{abstract}



\pagestyle{myheadings} \thispagestyle{plain}
\markboth{Z.C. Guo,  E.K.-W. Chu and X. Liang }{dADDA for Large-Scale M-Matrix Algebraic Riccati Equations}

\section{Introduction}\label{intro}

Consider the M-matrix  algebraic Riccati equation (MARE):
\begin{equation}\label{eq:mare}
XCX - XD - AX + B = 0, 
\end{equation}
where  
\[
W = \begin{bmatrix} D & -C \\ -B & A \end{bmatrix} \in \mathbb{R}^{(m + n) \times (m + n)} 
\]
is a nonsingular or an irreducible singular M-matrix, 
with $A\in \mathbb{R}^{m\times m}$, $D\in \mathbb{R}^{n\times n}$ and $B, C^{\T}, X \in \mathbb{R}^{m\times n}$. 
The MARE is a specific  type of  nonsymmetric algebraic Riccati equations, 
and  there have been many associated studies  from transport theory \cite{rbellmanG1975introduction,juang1995existence}. 
The solvability of a highly structured MARE, with rank-$1$ updated  $A$ and $D$, and rank-$1$ $B$ and $C$, 
has been established in \cite{juang1995existence,juangL1998nonsymmetric,juangN1995global}. 
A keen competition of various iterative methods then follows, mostly for the MARE with rank-$1$ structures  but also its generations, 
lasting more than twenty years. 
This involves various iteration schemes \cite{baiGL2008fast,chguoN2007iterative,guoL2000iterative,guoL2010convergence,lu2005solution,mehrmannX2008explicit}, 
Newton's method \cite{biniIP2008fast,linBW2008modified,lu2005newton}, transformation to quadratic equations \cite{biniMP2010transforming} 
and the structure-preserving doubling algorithm (SDA) \cite{chiangCGHLX2009convergence,guoLX2006structurepreserving} and 
its generalizations \cite{liCJL2011solution,liCJL2011solution:2dim,liCKL2013solving,wangWL2012alternatingdirectional,xueL2017highly}, 
with theories benefited  from the associated Hamiltonian matrix and M-matrix structures. 
The problem attracts such vast interests because of its interesting structures, as well as the diverse applications in 
transport theory \cite{guo2001nonsymmetric,juang1995existence,juangL1998nonsymmetric,juangN1995global,liCJL2011solution,liCJL2011solution:2dim}, Markov-modulated fluid queue theory \cite{glatoucheP2009stochastic,jdrogers1994fluid,xueL2017highly} and the Wiener-Hopf decompositions \cite{guo2001nonsymmetric}.  

Generally, the MARE~\eqref{eq:mare} admits more than one solutions \cite{lancasterR1995algebraic}   
owing to its nonlinearity. However, it is shown in \cite{guo2001nonsymmetric,guoL2000iterative} 
that~\eqref{eq:mare} has a unique minimal nonnegative solution  $X$ (interpreted componentwise) which is of interest in  practice. 
Here, by the minimal nonnegative solution we mean that $\Theta-X$ is nonnegative  for any other nonnegative solution $\Theta$ of the MARE.   
The dual problem of \eqref{eq:mare} has the form, with $Y\in \mathbb{R}^{n\times m}$:
\begin{equation}\label{eq:dual_nare}
	YBY - YA - DY + C=0. 
\end{equation}
The dual problem~\eqref{eq:dual_nare} is also an MARE and admits a unique minimal nonnegative solution $Y$. 

When solving the MAREs, one may particularly be interested in the accuracy of small entries in $X$,   
in how   small relative perturbations to the entries of $A, D, B$ and  $C$  
affect the entries in $X$. 
Are they small,  no matter how tiny the entries in $X$ are? 
Thanks to the componentwise perturbation analysis in~\cite{xueXL2012accurate},  
small elements do not possess larger relative errors. 
Hence if implemented carefully several methods can compute $X$ with high relative componentwise accuracy,   
including the fix-point iterations~\cite{guo2001nonsymmetric}, the Newton method~\cite{chguoN2007iterative}, 
the SDA~\cite{guoLX2006structurepreserving} and also the alternating-directional doubling algorithm (ADDA)~\cite{nguyenP2015componentwise,xueL2017highly}. 
All the methods mentioned above are efficient for MAREs of small or  medium sizes, in terms of execution time and memory requirements.              
Besides, in the SDA or ADDA,  the nonsingular M-matrix kernels may become ill-conditioned, 
especially for the critical case. Consequently, the inversions of these kernels may then cost the approximate solution its precision, and highly accurate computations for 
linear equations related to the kernels are necessary, which generally yields a high accuracy solution $X$.

In this paper, we solve the large-scale MAREs with low-rank structures, and 
propose a highly efficient method for the MARE \eqref{eq:mare}, 
generalizing the new decoupled form of the SDA \cite{guoCLL2020decoupled} and 
adapting the highly accurate  GTH-like algorithm for the M-matrix structures in the ADDA~\cite{xueL2017highly}. In detail, we show that the ADDA can be decoupled using the low-rank structures and for the solution $X$ one only needs to compute the iterative 
recursion of $H_k=\gamma \check{U}_k(I-Y_k Z_k)^{-1} \check{Q}_k^{\T}$. 
Avoiding calculating $E_k$ and $F_k$ (see \eqref{eq:sda}), the decoupled  
ADDA can be applied to solve large-scale problems efficiently because it 
takes full advantage of the structure that may exists in the original $W$. 
Then we prove that these small size kernels $I-Y_kZ_k$ are M-matrices, 
and we construct their triplet representations,  
for the GTH-like algorithm in high accuracy. 
Note that the highly accurate ADDA algorithm \cite{xueL2017highly} 
produces the triplet representations 
for some $n\times n$   and $m\times m$ M-matrices 
from that of $W$, computing  some required nonnegative vectors recursively. 
Our triplet representations are not straightforward adaptations of that given in~\cite{xueL2017highly}. With the decoupled form and the novel 
triplet representations for M-matrix kernels of  small sizes, 
we develop the highly accurate  algorithm (namely  dADDA)   
for large-scale MAREs, which takes $\bigO(m+n)$ flops per iteration when 
$A$ and $D$ are banded, sparse or low-rank updated of some diagonal matrices. 
Benefiting from the decoupling and the triplet representations of the kernels, 
as confirmed by our extensive experiments, the novel dADDA may 
presently be the most efficient algorithm for large-scale structured MAREs. 

The ADDA is decoupled using the low-rank structures in Section~\ref{ssec:decoupled-form}. 
We prove the kernels in the iterations are M-matrices in Section~\ref{ssec:m-matrices}, 
and  in Section~\ref{ssec:triple-representations} we construct their triplet representations. 
We summarize our algorithm  in Section~\ref{ssec:algorithm}. 
Illustrative numerical examples are presented in Section~\ref{sec:numerical-examples}, 
before we conclude in Section~\ref{sec:conclusions}.  

\subsection*{Notations}
By $\mathbb{R}^{n\times n}$ we denote the set of all $n\times n$ real matrices, 
with $\mathbb{R}^n = \mathbb{R}^{n\times 1}$ and $\mathbb{R}=\mathbb{R}^{1}$. 
The $n\times n$ identity matrix is $I_n$ and we write $I$ if its dimension is clear.  
The zero matrix is $0$ and  the superscript $(\cdot)^{\T}$ takes the transpose. 
By $\boldsymbol{1}_{l} \in \mathbb{R}^l$ and $\boldsymbol{1}_{k\times l} \in \mathbb{R}^{k\times l}$, 
respectively, we denote the $l$-vector and $k\times l$ matrix of all ones.  
The symbol  $M \otimes N$  is   the Kronecker product of 
the   matrices $M$ and $N$.  For $\Phi\in \mathbb{R}^{m\times n}$, 
$\Phi_{ij}$ is its $(i,j)$ entry, and by $|\Phi|$ we denote the matrix with elements  $|\Phi_{ij}|$. The inequality $\Phi \le \Psi$ 
holds if and only if $\Psi_{ij} \le  \Phi_{ij}$,  
and similarly for $\Phi < \Psi$, $\Phi \ge \Psi$ and $\Phi > \Psi$. 
In particular, $\Phi$ is a nonnegative matrix  means that $\Phi_{ij} \ge 0$. 
The submatrix of $\Phi$, comprised of the rows $k$ to $m$ and columns $l$ to $n$, is written as $\Phi_{k:m,l:n}$. 
We denote   the triplet representation of $W$ by  
\[
(N_W, 
	[u_1^{\T}, u_2^{\T}]^{\T}, 
[v_1^{\T}, v_2^{\T}]^{\T}), 
\]
where $N_W$ is the off-diagonal part of $-W$, with  
\[
	W \begin{bmatrix}
		u_1 \\ u_2
	\end{bmatrix} = \begin{bmatrix}
		v_1 \\ v_2
	\end{bmatrix} \ge 0, 
\]
$0<u_1\in \mathbb{R}^n$, $0<u_2\in \mathbb{R}^m$, 
$0\le v_1\in \mathbb{R}^{n}$ and $0\le v_2\in \mathbb{R}^m$.

\section{Decoupled ADDA with high accuracy}\label{sec:decoupled-adda-with-high-accuracy}

The highly accurate   alternating-directional doubling algorithm  \cite{xueL2017highly} (accADDA)   is  a variation of the SDA and shares the same doubling recursions:
\begin{equation}\label{eq:sda}
	\begin{aligned}
F_{k+1} &= F_k (I_m - H_k G_k)^{-1} F_k, \ \ \ E_{k+1} = E_k (I_n - G_k H_k)^{-1} E_k, \\
H_{k+1} &= H_k + F_k (I_m-H_k G_k)^{-1} H_k E_k, \ \ \ G_{k+1} = G_k + E_k (I_n-G_k H_k)^{-1} G_k F_k.  
	\end{aligned}
\end{equation}
The pivotal difference from the SDA  lies in the initial iterates 
with the  two parameters $\alpha$ and $\beta$, with $D_\alpha := \alpha D + I$ 
and $A_\beta := \beta A + I$:
\begin{align}\label{eq:initial-matrix}
	\begin{bmatrix}
		E_0 & G_0 \\
		H_0 & F_0
	\end{bmatrix} = \begin{bmatrix}
	D_\alpha & -\beta C
	\\
	-\alpha B & A_\beta
\end{bmatrix}^{-1}
\begin{bmatrix}
	D_{-\beta} & \alpha C\\
	\beta B  & A_{-\alpha}
\end{bmatrix},
\end{align}
where $0\le \alpha \le \min_i a_{ii}^{-1}$,  $0 \le \beta\le \min_j d_{jj}^{-1}$, $\max\{\alpha, \beta\} > 0$,  
with $a_{ii}$ and $d_{jj}$ respectively being the diagonal entries of $A$ and $D$. The nonsingularity of the matrix in \eqref{eq:initial-matrix}  
is  guaranteed as below.  

\begin{lemma}[{\cite[Lemma~3.1]{xueL2017highly}}]\label{lm:initial-M}
	Let $\alpha\ge 	0$ and $\beta \ge 0$ with $\max\{\alpha, \beta\}>0$, then $\begin{bmatrix}
		D_\alpha & -\beta C \\ -\alpha B & A_\beta 
	\end{bmatrix}$ is a nonsingular M-matrix. 
\end{lemma}

With $(F_0, E_0, H_0, G_0)$ from \eqref{eq:initial-matrix} as the initial iterates, 
by the doubling recursions in \eqref{eq:sda}, 
\cite{xueL2017highly}  shows that  
the sequences $(F_k, E_k, H_k, G_k)$ satisfy   
\begin{enumerate}
	\item $E_k\ge 0$ and $F_k\ge 0$  are uniformly bounded with respect to $k$; 
	\item $I-H_k G_k$ and $I-G_k H_k$ are nonsingular M-matrices; and 
	\item $0\le H_k \le H_{k+1} \le X$, $0\le G_k \le G_{k+1} \le Y$, implying that $\{H_k\}$ and $\{G_k\}$ respectively  converge increasingly to $X$ and $Y$. 
\end{enumerate}
Actually, $\{H_k\}$ and $\{G_k\}$  converge quadratically except for the critical case \cite{chiangCGHLX2009convergence}, 
in which they  converge linearly with a rate of $0.5$. 
With the uniformly bounded property, Xue and Li~\cite{xueL2017highly} subtly  devised the triplet representations of the nonsingular M-matrices $I-H_k G_k$ and $I-G_k H_k$ by tracking the difference 
$\begin{bmatrix}
	u_1 \\ u_2 
\end{bmatrix} - \begin{bmatrix}
E_k & G_k \\ H_k & F_k
\end{bmatrix}\begin{bmatrix}
	u_1 \\ u_2
\end{bmatrix}$ in the accADDA for MAREs.  

Recursively,  one constructs the triplet  representations of the M-matrices kernels $I-H_k G_k$ and $I-G_k H_k$ in a cancellation-free manner~\cite{xueL2017highly}.   
However, it is not suitable for large-scale MAREs with low-rank structures   since all the iterates $F_k$, $E_k$, $H_k$ and $G_k$ 
are required, leading to  a computational complexity of $\bigO(n^3+m^3)$ in each iteration.    
To adopt the  accADDA  to  large-scale MAREs, we firstly show that it can be decoupled when low-rank structure exists, as in our new dADDA. 
We prove the kernels in the dADDA are nonsingular M-matrices of small sizes. 
Then  we construct the triplet representations of these  nonsingular kernels,   
enabling the highly accurate  
computation with the GTH-like algorithm  for 
 the associated linear equations.    

We  assume that $B$ and $C$ are of low-rank with the full rank factorizations $B = B_l B_r^{\T}$ and $C = C_l C_r^{\T}$, 
where $0 \le B_l \in \mathbb{R}^{m\times p}$, $0\le B_r \in \mathbb{R}^{n\times p}$, $0\le C_l \in \mathbb{R}^{n\times q}$ and $0\le C_r\in \mathbb{R}^{m\times q}$.

\subsection{Decoupled ADDA}\label{ssec:decoupled-form}

Firstly, the initial iterates, specified in \eqref{eq:initial-matrix},  
can  be rewritten  as  
\begin{equation}\label{eq:initial}
	\begin{aligned}
		F_0& = (A_\beta  - \alpha \beta B D_\alpha^{-1} C)^{-1} (A_{-\alpha} + \alpha^2B D_\alpha^{-1}C),
		\\
		E_0 &= (D_\alpha - \alpha \beta C A_\beta^{-1}B)^{-1}(D_{-\beta} + \beta^2 C A_\beta^{-1}B), 	
		\\
		H_0 & = \gamma (A_\beta - \alpha \beta B D_\alpha^{-1} C)^{-1}BD_\alpha^{-1}, \qquad 
		G_0  = \gamma (D_\alpha - \alpha \beta C A_\beta^{-1}B)^{-1} CA_\beta^{-1},
	\end{aligned}
\end{equation}
where  $\gamma := \alpha + \beta>0$  and the nonsingularity of 
$D_\alpha$ and $A_\beta$ follows from that of $W$.  
In fact, since $W$ is a nonsingular or an irreducible M-matrix,  
$A$ and $D$ are nonsingular M-matrices, implying  that  $D_\alpha$ and $A_\beta$  are  nonsingular.  
Furthermore, by Lemma~\ref{lm:initial-M} and the results in \cite{meyer1989stochastic}, 
we have further results for the Schur complements    
$A_\beta - \alpha \beta BD_\alpha^{-1} C$ and $D_\alpha - \alpha \beta C A_\beta^{-1}B$, given below. 

\begin{lemma}\label{lm:schur-complement}
Let $\alpha\ge 	0$ and $\beta \ge 0$ with $\max\{\alpha, \beta\}>0$, 
then $D_\alpha$ and $A_\beta$ and  the Schur complements 
$A_\beta - \alpha \beta BD_\alpha^{-1} C$ and $D_\alpha - \alpha \beta C A_\beta^{-1}B$ 
are nonsingular M-matrices.
\end{lemma}

Now substitute the full rank factorizations $B=B_lB_r^{\T}$ and $C=C_lC_r^{\T}$ 
into \eqref{eq:initial}, by the Sherman-Morrison-Woodbury formula (SMWF), we obtain   
\begin{equation}\label{eq:schur1}
	\begin{aligned}
		&(A_\beta - \alpha  \beta B D_\alpha^{-1}C)^{-1}
		=(A_\beta - \alpha  \beta B_lB_r^{\T} D_\alpha^{-1}C_l C_r^{\T})^{-1}
		\\
		=&A_\beta^{-1} + \alpha \beta A_\beta^{-1} B_l (I - \alpha \beta B_r^{\T} D_\alpha^{-1} C_l C_r^{\T} A_\beta^{-1} B_l )^{-1} B_r^{\T} D_\alpha^{-1} C_l C_r^{\T} A_\beta^{-1} 
		\\
		=& A_\beta^{-1} + \beta U_0 (I-Y_0 Z_0)^{-1} Y_0 V_0^{\T}, 
	\end{aligned}
\end{equation}
\begin{equation}\label{eq:schur2}
	\begin{aligned}
		&	(D_\alpha - \alpha \beta C A_\beta^{-1}B )^{-1}
		= 	(D_\alpha- \alpha \beta C_lC_r^{\T} A_\beta^{-1}B_l B_r^{\T} )^{-1}
		\\
		=&D_\alpha^{-1} + \alpha \beta D_\alpha^{-1}C_l (I-\alpha \beta C_r^{\T} A_\beta^{-1}B_l B_r^{\T}D_\alpha^{-1}C_l )^{-1}C_r^{\T} A_\beta^{-1}B_l B_r^{\T} D_\alpha^{-1}
		\\
		=& D_\alpha^{-1} + \alpha W_0 (I-Z_0 Y_0)^{-1}Z_0 Q_0^{\T},
	\end{aligned}
\end{equation}
\begin{align*}
	\alpha \left[ A_\beta^{-1} + \beta U_0 (I-Y_0 Z_0)^{-1}Y_0 V_0^{\T}\right]
	B_l B_r^{\T} D_\alpha^{-1} C_l C_r^{\T}
	&=
	U_0(I-Y_0 Z_0)^{-1} Y_0 C_r^{\T},	
	\\
	\beta \left[D_\alpha^{-1} + \alpha W_0(I-Z_0 Y_0)^{-1}Z_0 Q_0^{\T}\right] 
	C_l C_r^{\T} A_\beta^{-1} B_l B_r^{\T}
	&=
	W_0(I-Z_0 Y_0)^{-1} Z_0 B_r^{\T},
\end{align*}
where 
\begin{equation}\label{eq:initial-decoupled-1}
	\begin{aligned}
		&U_0 = A_\beta^{-1}B_l,\quad  V_0 = A_\beta^{-\T}C_r, \quad  W_0 = D_\alpha^{-1} C_l, \quad   Q_0 =D_\alpha^{-\T} B_r,    
		\\
	&Y_0 = \alpha B_r^{\T} D_\alpha^{-1} C_l, \quad  Z_0 =\beta C_r^{\T} A_\beta^{-1}B_l.
	\end{aligned}
\end{equation}
As a result, after  defining 
\begin{equation}\label{eq:initial-decoupled-2}
	A_{\alpha,\beta} = A_\beta^{-1} A_{-\alpha}, \qquad 
	D_{\alpha,\beta} = D_\alpha^{-1} D_{-\beta}, 
\end{equation}
we obtain  
\begin{eqnarray*}
	F_0 
	& =&
	A_\beta^{-1} A_{-\alpha} + \beta U_0(I-Y_0 Z_0)^{-1} Y_0 V_0^{\T} A_{-\alpha} + \alpha U_0 (I-Y_0 Z_0)^{-1} Y_0 C_r^{\T}
	\\
	&\equiv& A_{\alpha,\beta} + \gamma U_0(I-Y_0 Z_0)^{-1} Y_0 V_0^{\T}, \\
	E_0 
	& =&
	D_\alpha^{-1} D_{-\beta} + \alpha W_0(I-Z_0 Y_0)^{-1} Z_0 Q_0^{\T} D_{-\beta} + \beta W_0 (I-Z_0 Y_0)^{-1} Z_0 B_r^{\T}
	\\
	&\equiv& D_{\alpha,\beta} + \gamma W_0(I-Z_0 Y_0)^{-1} Z_0 Q_0^{\T}, \\
	H_0 &=& \gamma \left[
	A_\beta^{-1} + \beta U_0 (I-Y_0Z_0)^{-1}Y_0 V_0^{\T} \right] 
	B_l B_r^{\T}D_\alpha^{-1}
	 \equiv \gamma U_0 (I-Y_0 Z_0)^{-1} Q_0^{\T}, \\
	G_0 & =& \gamma  
	\left[ D_\alpha^{-1} + \alpha W_0 (I-Z_0 Y_0)^{-1} Z_0 Q_0^{\T}\right]
	C_l C_r^{\T}A_\beta^{-1}
	\equiv \gamma W_0(I-Z_0 Y_0)^{-1} V_0^{\T}.
\end{eqnarray*}
Apparently, $H_0$ and $G_0$ are decoupled. 

Next define  $S_0 := V_0^{\T}U_0$, $T_0 := Q_0^{\T} W_0$, $U_1 := A_{\alpha,\beta}U_0$, $Q_1 := D_{\alpha,\beta}^{\T} Q_0$ and  
$L := I-Y_0Z_0 - \gamma^2 T_0(I-Z_0Y_0)^{-1}S_0$, then we get   
\begin{align*}
	&(I-H_0 G_0)^{-1} = \left[I-\gamma^2 U_0(I-Y_0Z_0)^{-1} T_0 (I-Z_0 Y_0)^{-1} V_0^{\T}\right]^{-1}
	\equiv  I + \gamma^2 U_0 L^{-1} T_0 (I-Z_0 Y_0)^{-1} V_0^{\T},
	\\
	&(I-H_0G_0)^{-1}H_0 
	\equiv \gamma U_0 L^{-1} Q_0^{\T}.
\end{align*}
Consequently,   it holds that 
\begin{align*}
	H_1 
	&= H_0 + F_0(I-H_0 G_0)^{-1}H_0 E_0 
	\\
	&=
	\gamma 
		[U_0,\, U_1]
	\begin{bmatrix}
		I & \gamma (I-Y_0Z_0)^{-1}Y_0 S_0 \\ 0 & I 
	\end{bmatrix}
	\begin{bmatrix}
		I-Y_0 Z_0 & \\ & L
	\end{bmatrix}^{-1}
	\begin{bmatrix}
		I & 0 \\ \gamma T_0 (I-Z_0 Y_0)^{-1}Z_0 & I
	\end{bmatrix}
	\begin{bmatrix}
		Q_0^{\T} \\ \\ Q_1^{\T}
	\end{bmatrix}
	\\
	& \equiv \gamma 
		[U_0,\, U_1] 
	(I - Y_1Z_1)^{-1}
	\begin{bmatrix}
		Q_0^{\T} \\ \\ Q_1^{\T}
	\end{bmatrix},
\end{align*}
where 
	$Y_1 = \begin{bmatrix}
		0 & Y_0 \\ Y_0 & \gamma T_0
	\end{bmatrix}$ and 
	$Z_1 =\begin{bmatrix}
		0& Z_0 \\ Z_0 & \gamma S_0
	\end{bmatrix}$. 
Now  denote    
$V_1 := A_{\alpha,\beta}^{\T}V_0$ and $W_1 := D_{\alpha,\beta}W_0$, 
then a similar but tedious process produces  
\begin{align*}
	G_1& = \gamma 
		[W_0,\, W_1]
	(I - Z_1 Y_1)^{-1} 
	\begin{bmatrix}
		V_0^{\T} \\ \\ V_1^{\T}
	\end{bmatrix}, 
	\qquad 
	F_1  = A_{\alpha,\beta}^2 + \gamma 
		[U_0,\, U_1] 
	(I-Y_1 Z_1)^{-1} Y_1 \begin{bmatrix}
		V_0^{\T} \\ \\ V_1^{\T}
	\end{bmatrix},
	\\
	E_1 & = D_{\alpha,\beta}^2 + \gamma 
		[W_0,\, W_1] 
	(I-Z_1 Y_1)^{-1} Z_1 
	\begin{bmatrix}
		Q_0^{\T} \\ \\ Q_1^{\T}
	\end{bmatrix}.
\end{align*}

Clearly, $H_1$ and $G_1$ are decoupled. By  a similar process with the help of the SMWF, we eventually obtain the sequences $(F_k, E_k, H_k, G_k)$ in the dADDA, as in the following theorem. 

\begin{theorem}[dADDA]\label{thm:decoupled-form-for-the-adda}
	Let  $U_j:=A_{\alpha,\beta} U_{j-1}$, $V_j:=A_{\alpha,\beta}^{\T} V_{j-1}$, 
	$W_j:=D_{\alpha,\beta} W_{j-1}$ and $Q_j:= D_{\alpha,\beta}^{\T} Q_{j-1}$ for $j\ge 1$.  For $k\ge 1$,   denote 
	\begin{align*}
		\check{U}_k &= [U_0, U_1, \cdots, U_{2^k-1}], \hspace*{1.05cm} 
		\check{V}_k =[V_0, V_1, \cdots, V_{2^k-1}],
		\\
		\check{W}_k &= [W_0, W_1, \cdots, W_{2^k-1}], \qquad
		\check{Q}_k =[Q_0, Q_1, \cdots, Q_{2^k-1}], 
	\end{align*}
	and  let 
	\begin{align}\label{eq:YZ}
		Y_k=\begin{bmatrix}
			0 & Y_{k-1} \\ Y_{k-1} & \gamma T_{k-1}
		\end{bmatrix}, 
		\qquad 
		Z_k=\begin{bmatrix}
			0&  Z_{k-1}\\ Z_{k-1} & \gamma S_{k-1}
		\end{bmatrix}
	\end{align} 
	with $T_{k-1} = \check{Q}_{k-1}^{\T} \check{W}_{k-1}$ and $S_{k-1} = \check{V}_{k-1}^{\T} \check{U}_{k-1}$. 
	Then the  iteration in \eqref{eq:sda} has the 
	following decoupled form
	\begin{align*}
		& F_k = A_{\alpha,\beta}^{2^k} + \gamma  \check{U}_k  
		(I  - Y_k Z_k)^{-1} Y_k  \check{V}_k^{\T}, 
		\qquad 
		E_k = D_{\alpha,\beta}^{2^k} + \gamma \check{W}_k
		(I-Z_k Y_k)^{-1} Z_k \check{Q}_k^{\T}, 
		\\
		& H_k = \gamma \check{U}_k 
		(I-Y_k Z_k)^{-1} \check{Q}_k^{\T},
		\qquad \qquad \qquad  
		G_k = \gamma \check{W}_k
		(I-Z_k Y_k)^{-1} \check{V}_k^{\T}. 
	\end{align*}
\end{theorem}

The proof of Theorem~\ref{thm:decoupled-form-for-the-adda} 
is similar to that for 
\[
	(\widehat{F}_k, \widehat{E}_k, \widehat{H}_k, \widehat{G}_k) \equiv 
	\left((\alpha^{-1} \beta)^{2^k} F_k, (\beta^{-1} \alpha)^{2^k}E_k, H_k, G_k\right)
\]
given in~\cite{guoCLL2020decoupled}, which we omit.  	

\subsection{Kernels are M-matrices}\label{ssec:m-matrices}

When applying the dADDA  to solve  large-scale MAREs with low-rank structures 
to obtain highly  accurate solution, 
two  crucial issues  have to be settled: 
\begin{enumerate}
\item Are $U_j$, $V_j$, $W_j$ and $Q_j$ all nonnegative? 
\item Are the kernels $I-Y_k Z_k$ and $I-Z_k Y_k$  nonsingular M-matrices? 
\end{enumerate}
For both questions, we need to show that $A_{\alpha,\beta}$, $D_{\alpha,\beta}$, 
$(I-Y_k Z_k)^{-1}$ and $(I-Z_k Y_k)^{-1}$ are nonnegative. 
In the following, we  assume that $\max\{\alpha,\beta\}>0$, which can be satisfied by choice.

\begin{lemma}\label{lm:ADnonnegative}
It holds that $A_{-\alpha} \ge0$ and $D_{-\beta} \ge 0$.	
\end{lemma}
\begin{proof}
	Since $W$ is a nonsingular or an irreducible singular M-matrix, 
	then $A$ is a nonsingular M-matrix, implying  $a_{ij}\le 0$ with $i\neq j$ and $a_{ii} > 0$. 
	Thus  in $I-\alpha A$, the off-diagonal entries $-\alpha a_{ij}$ are nonnegative. For the diagonal elements, because  $\alpha \le \min_i a_{ii}^{-1} = \left(\max_i a_{ii}\right)^{-1}$, then we have 
	$\alpha a_{ii} \le  a_{ii} \left(\max_i a_{ii}\right)^{-1} \le 1$, showing that $1-\alpha a_{ii} \ge 0$. Hence $A_{-\alpha}\equiv I-\alpha A$ is nonnegative. Similarly, we can show that $D_{-\beta} = I-\beta D$ is nonnegative. 	
\end{proof}

The nonnegativity of $B_l$, $B_r$, $C_l$, $C_r$, $A_\beta^{-1}$ and $ D_\alpha^{-1}$ leads to that of $U_0$,  $ Q_0$, $W_0$, $V_0$, $Y_0$ and  $Z_0$. Then by Lemmas~\ref{lm:schur-complement} and~\ref{lm:ADnonnegative}, we know that 
$A_{\alpha, \beta}= A_\beta^{-1}A_{-\alpha} \ge 0$ and $D_{\alpha,\beta}=D_\alpha^{-1} D_{-\beta} \ge 0$.  
Furthermore, it holds that  
$U_k=A_{\alpha,\beta}^{k} U_{0}\ge 0$, 
$V_k=(A_{\alpha,\beta}^{\T})^k V_{0}\ge 0$, 
$W_k=D_{\alpha,\beta}^{k} W_{0} \ge 0$ and $Q_k= (D_{\alpha,\beta}^{\T})^k Q_{0} \ge 0$. 
Also, we have   $S_k\ge 0$, $T_k\ge 0$, $Y_k\ge0$ and $Z_k\ge 0$. 

\begin{lemma}\label{lm:M-matrix-initial}
	The kernels  $I-Y_0Z_0$ and $I-Z_0 Y_0$ are nonsingular M-matrices. 
\end{lemma}

\begin{proof}
It follows from Lemma~\ref{lm:schur-complement} that 
\begin{align*}
	(A_\beta- \alpha  \beta B D_\alpha^{-1}C)^{-1} \ge 0, \quad 
	(D_\alpha - \alpha \beta C A_\beta^{-1}B)^{-1}\ge 0. 
\end{align*}
Furthermore, by~\eqref{eq:schur1} and~\eqref{eq:schur2}   
it  holds that 
\begin{align*}
	0 &\le \beta C_r^{\T} (A_\beta - \alpha  \beta B D_\alpha^{-1}C)^{-1} B_l
	= \beta C_r^{\T}\left[A_\beta^{-1} + \beta U_0 (I-Y_0 Z_0)^{-1} Y_0 V_0^{\T} \right] B_l
	\\
	& = Z_0  + Z_0(I-Y_0 Z_0)^{-1} Y_0 Z_0 
	\equiv Z_0 (I-Y_0 Z_0)^{-1}, 
	\\
	0 & \le \alpha B_r^{\T} (D_\alpha - \alpha \beta C A_\beta^{-1}B)^{-1} C_l
	= \alpha B_r^{\T} \left[ D_\alpha^{-1} + \alpha W_0 (I-Z_0 Y_0)^{-1}Z_0 Q_0^{\T}\right]C_l
	\\
	&= Y_0 + Y_0(I-Z_0 Y_0)^{-1} Z_0 Y_0
	\equiv Y_0(I-Z_0 Y_0)^{-1}.
\end{align*}
Because $Y_0\ge 0$ and $Z_0\ge 0$,  we have $Y_0 Z_0 (I-Y_0 Z_0)^{-1}\ge 0$ 
and $Z_0Y_0(I-Z_0 Y_0)^{-1}\ge 0$, implying  
\begin{align*}
	I + Y_0 Z_0 (I-Y_0 Z_0)^{-1} \equiv (I-Y_0 Z_0)^{-1}\ge 0, \quad \,  
	I + Z_0 Y_0(I-Z_0 Y_0)^{-1}  \equiv (I-Z_0 Y_0)^{-1} \ge 0.
\end{align*}
Hence $I-Y_0Z_0$ and $I-Z_0 Y_0$ are nonsingular M-matrices 
due to the fact that a nonsingular Z-matrix is an M-matrix if and 
only if its inverse is nonnegative~\cite{bermanP1994nonnegative}. 
\end{proof}

The following theorem concerns the kernels  $I-Y_kZ_k$ and $I-Z_kY_k$  in the dADDA. 

\begin{theorem}\label{thm:m-matrix}
The kernels  $I-Y_k Z_k$ and $I-Z_kY_k$ are nonsingular M-matrices for all $k\ge 0$.
\end{theorem}
\begin{proof}
	We  prove  by induction, with the case for $k=0$ from Lemma~\ref{lm:M-matrix-initial}. Assume that the result holds for $k\ge 1$, that is $(I-Y_k Z_k)^{-1}\ge 0$ and $(I-Z_kY_k)^{-1}\ge 0$,  
	we then  show  that $(I-Y_{k+1} Z_{k+1})^{-1}\ge 0$. 
	Since $I-H_k G_k$ is a nonsingular M-matrix, then with  
	\[
		K := \left[I- \gamma^2 (I-Y_k Z_k)^{-1}T_k(I-Z_k Y_k)^{-1}S_k\right]^{-1}, 
	\]
	we have  
	\begin{align*}
		0 &\le (I-H_k G_k)^{-1}
		= \left[I - \gamma^2 \check{U}_k(I-Y_k Z_k)^{-1}T_k(I-Z_k Y_k)^{-1}\check{V}_k^{\T} \right]^{-1}
		\\
		&= I + \gamma^2 \check{U}_k\left[I - \gamma^2 (I-Y_k Z_k)^{-1}T_k(I-Z_k Y_k)^{-1}\check{V}_k^{\T}\check{U}_k\right]^{-1} (I-Y_k Z_k)^{-1}T_k(I-Z_k Y_k)^{-1}\check{V}_k^{\T} 
		\\
		&=I + \gamma^2 \check{U}_k\left[I-\gamma^2(I-Y_k Z_k)^{-1}T_k(I-Z_k Y_k)^{-1}S_k\right]^{-1}(I-Y_k Z_k)^{-1}T_k(I-Z_k Y_k)^{-1}\check{V}_k^{\T}
		\\
		&\equiv 
		I+\gamma^2\check{U}_k K (I-Y_kZ_k)^{-1} T_k(I-Z_kY_k)^{-1}\check{V}_k^{\T},
		\\
		0&\le \check{V}_k^{\T}(I-H_k G_k)^{-1} \check{U}_k
		=\check{V}_k^{\T}\left[ I+\gamma^2\check{U}_k K (I-Y_kZ_k)^{-1} T_k(I-Z_kY_k)^{-1}\check{V}_k^{\T}\right]
		\check{U}_k
		\\
		& = S_k +\gamma^2 S_k K  
		(I-Y_kZ_k)^{-1}T_k(I-Z_kY_k)^{-1}S_k
		\equiv  S_k K.
	\end{align*}
	Moreover, it follows from $(I-Y_kZ_k)^{-1}\ge 0$, $T_k=\check{Q}_k^{\T}\check{W}_k\ge 0$ and  $(I-Z_k Y_k)^{-1}\ge 0$ that 
	\begin{align*}
		0\le I+	\gamma^2 (I-Y_k Z_k)^{-1}T_k(I-Z_k Y_k)^{-1}S_k K
		&\equiv  K, 
	\end{align*}
	leading to 
	\begin{align*}
		M:=\left[I-Y_kZ_k - \gamma^2 T_k(I-Z_k Y_k)^{-1}S_k \right]^{-1}  \equiv  K (I-Y_k Z_k)^{-1} \ge 0. 
	\end{align*}

	Substituting  the expression~\eqref{eq:YZ} for $Y_{k+1}$ and $Z_{k+1}$ 
	into $(I-Y_{k+1}Z_{k+1})^{-1}$, we obtain    
	\begin{align*}
		&(I-Y_{k+1} Z_{k+1})^{-1}
		= \begin{bmatrix}
			I-Y_kZ_k & - \gamma Y_k S_k 
			\\
			- \gamma T_kZ_k & I-Y_k Z_k - \gamma^2 T_kS_k
		\end{bmatrix}^{-1}
		\\
		=& \begin{bmatrix}
			I-Y_kZ_k & - \gamma Y_kS_k
			\\
			0&I-Y_kZ_k- \gamma^2 T_k(I-Z_kY_k)^{-1}S_k
		\end{bmatrix}^{-1}
		\begin{bmatrix}
			I & 0\\
			\gamma T_kZ_k(I-Y_kZ_k)^{-1} & I
		\end{bmatrix}
		\\
		=&\begin{bmatrix}
			(I-Y_kZ_k)^{-1} &\gamma (I-Y_kZ_k)^{-1}Y_kS_kM
			\\
			0 & M
		\end{bmatrix}
		\begin{bmatrix}
			I & 0\\
			\gamma T_kZ_k(I-Y_kZ_k)^{-1} & I
		\end{bmatrix}.
	\end{align*}
	Since $(I-Y_k Z_k)^{-1}\ge 0$, $Y_k\ge 0$, $S_k\ge0$, 
	$M\ge 0$, $T_k\ge 0$ and $Z_k\ge 0$, then   
	$(I-Y_{k+1} Z_{k+1})^{-1}\ge 0$, implying that $I-Y_{k+1}Z_{k+1}$ is a nonsingular M-matrix, thus so is  $I-Z_{k+1}Y_{k+1}$.

\end{proof}

\subsection{GTH-like algorithm and triplet representations}\label{ssec:triple-representations}

Firstly, we briefly sketch the GTH-like algorithm presented in~\cite{alfaXY2002accurate}, which solves the M-matrix linear system $M x=b$ in high accuracy, 
where each entry of the solution  $x$   have almost full relative accuracy. 
Given the triplet representation of the nonsingular M-matrix $M$,  
the GTH-like algorithm, 
a variation of  elementary row operations in the  Gaussian elimination  
without pivoting, computes the LU factorization  
with  high relative componentwise accuracy   
since computations are cancellation-free.   

Let $M\in \mathbb{R}^{n\times n}$ be a nonsingular M-matrix and $(N_M, u_M, v_M)$ be its triplet representation, where $N_M$ is the off-diagonal part of $-M$, $u_M>0$ and $v_M = M u_M\ge 0$. Obviously, we have $N_M\ge 0$ and the diagonal part of $M$ can be determined in a cancellation-free way: 
\begin{align}\label{eq:Mii}
	M_{ii} = \frac{(v_M)_i + \sum_{j\neq i} (N_M)_{ij} (u_M)_j}{(u_M)_i}.
\end{align}
It can be verified that  $M^{(k)}\in \mathbb{R}^{(n-k)\times(n-k)}$, the coefficient matrix after $k$ Gaussian eliminations, is still a nonsingular M-matrix. Moreover, the triplet representation of $M^{(k)}$ can be constructed from that of $M^{(k-1)}$, with $M^{(0)} = M$. 
As a result, based on~\eqref{eq:Mii}, one can  compute the LU factorization of $M$  cancellation-free. 
We outline the GTH-like algorithm from~\cite{alfaXY2002accurate} in Algorithm~\ref{alg:GTH-like}.

\begin{algorithm}[h]
	\caption{GTH-like algorithm for solving $Mx=b$}\label{alg:GTH-like}
	\hspace*{0.02in}{\bf Input:}
	the triplet representation $(N_M,u_M, v_M)$ and vector $b$. 
	\\
	\hspace*{0.02in}{\bf Output:}
	$x=M^{-1}b$. 
	\begin{algorithmic}[1]
		\State set $L=I_n$, $U=-N_M\le 0$;
		\For{$k=1:1:n$}
		\State $U_{k,k} =\left[ (v_M)_k - U_{k,k+1:n}(u_M)_{k+1:n}\right] / (u_M)_k$;
		\State $L_{k+1:n} = U_{k+1:n} / U_{k,k}$;
		\State $U_{k+1:n,k} = 0$;
		\State $U_{k+1:n,k+1:n} = U_{k+1:n,k+1:n} -L_{k+1:n,k} U_{k,k+1:n}$;
		\State set the diagonal of $U_{k+1:n, k+1:n}$ as $0$;
		\State $(v_M)_{k+1:n} = (v_M)_{k+1:n} - (v_M)_k L_{k+1:n,k}$;
		\\
		\EndFor{\bf end}
		\State solve $Ly=b$ with forward substitution;
		\State solve $Ux=y$ with backward substitution.
	\end{algorithmic}
\end{algorithm}

Note that when $b\ge 0$, no subtraction occurs in the forward and backward substitutions, thus the whole  solution process is  cancellation-free, leading to  full accuracy for all  entries of $x$. For the detailed analysis for the GTH-like algorithm, please refer to~\cite{alfaXY2002accurate, xueXL2012accurateSylvester}.      

The computational complexity of Algorithm~\ref{alg:GTH-like} is $\bigO(n^3)$  and the dominant cost lies in line~6. Hence Algorithm~\ref{alg:GTH-like} is efficient for M-matrix linear systems of medium sizes. 
However, for large-scale problems with some special structures 
	like a banded matrix  or a rank-one update of a nonsingular diagonal matrix,  its complexity may be reduced.  
For example, where $M$ is banded with the maximum number of nonzero elements on each row and column being $c$, the computational complexity of Algorithm~\ref{alg:GTH-like} will be reduced to $\bigO(cn)$.

\begin{remark}\label{rk:GTH-for-lowrank}
	Let  $M=D_M - ab^{\T}$ with $D_M$ being diagonal and $a, b>0$,  then it  takes $\bigO(n^2)$ flops to get $x$ applying Algorithm~\ref{alg:GTH-like}:  the  apparent computational complexity for the LU factorization is $\bigO(n^2)$ since for all $k\ge0$ the off-diagonal part of $M^{(k)}$ are rank-one updates with some diagonal matrices; the forward and backward steps obviously require $\bigO(n^2)$ flops. In fact, for large-scale problems such complexity is far from satisfied. 
	Fortunately, the Sherman-Morrison formula  can  provide  a perfect remedy: 
	it follows from $v_M = D_Mu_M-(b^{\T}u_M)a$ that  $b^{\T}u_M(1-b^{\T}D_M^{-1}a) = b^{\T}D_M^{-1}v_M$, suggesting $1-b^{\T}D_M^{-1}a = \frac{b^{\T}D_M^{-1}v_M}{b^{\T}u_M}$; then it holds that 
	\begin{align*}
		M^{-1}&=
		D_M^{-1} + \frac{1}{1-b^{\T}D_M^{-1}a}(D_M^{-1}a)(b^{\T}D_M^{-1})
		\equiv
		D_M^{-1} + \frac{b^{\T}u_M}{b^{\T}D_M^{-1}v_M} (D_M^{-1}a)(D_M^{-1}b)^{\T},
	\end{align*}
indicating that the whole process is cancellation-free  and  thus highly accurate.   More importantly, the complexity is reduced to $\bigO(n)$.  In fact, the above technique can be  extended to the  low-rank structures, that is, $M=D_M + U_MV_M^{\T}$ with $D_M$ being diagonal and $U_M, V_M\in \mathbb{R}^{n\times r}$, $r\ll n$, 
	and the complexity remains $\bigO(n)$. 
	The computational process is  similar to that for $r=1$  and we omit the details.  
\end{remark}

In Section~\ref{ssec:m-matrices}, we demonstrate that the kernels $I-Y_kZ_k$ and $I-Z_kY_k$ are nonsingular M-matrices. 
To solve the associated M-matrix linear  systems with the GTH-like algorithm \cite{alfaXY2002accurate,nguyenP2015componentwise}, one needs 
 the triplet representations for  those kernels. 

From the triplet representation of $W$, we have $W \begin{bmatrix}
	u_1\\  u_2
	\end{bmatrix}=\begin{bmatrix}
	v_1 \\  v_2
\end{bmatrix}$, or equivalently  
\begin{align*}
	D u_1 = v_1 + C u_2, \qquad 
	A u_2 = v_2 + B u_1,
\end{align*}
which leads to 
\begin{align}\label{eq:triple-repreAD}
	D_\alpha u_1 = \alpha v_1 +  u_1+ \alpha Cu_2\ge  0, \qquad 
	A_\beta u_2 =\beta v_2 +  u_2 + \beta B u_1\ge 0.
\end{align}
Consequently, with $N_{\Theta}=\mathrm{diag}(\Theta)-\Theta$, 
\eqref{eq:triple-repreAD} gives  the triplet representations of the nonsingular M-matrices  $D_\alpha$ and $A_\beta$, respectively: 
\begin{align}\label{eq:triple-repreAD1}
	(N_{D_\alpha}, u_1,\alpha v_1+ u_1 + \alpha C u_2),  
	\qquad 
	(N_{A_\beta}, u_2, \beta v_2 +  u_2 +\beta Bu_1). 
\end{align}
These triplet representations enable the 
GTH-like algorithm \cite{alfaXY2002accurate,nguyenP2015componentwise} to calculate 
$U_j, V_j, W_j, Q_j$ and $T_j, S_j$, $Y_j$, $Z_j$ in high accuracy without cancellations, 
for $j\ge 0$. Recall from the structure of $W$ and the properties of M-matrices, it is obvious that $N_{D_\alpha}$ and $N_{A_\beta}$ are nonnegative. 

\begin{remark}\label{rk:GTH-trans}
	Note that the GTH-like algorithm works via the LU 
	factorizations of $D_\alpha$ and $A_\beta$, 
	which can be applied to obtain 	$Q_j$ and $V_j$ accurately. 
\end{remark}

\begin{theorem}\label{thm:triple-representation-initial}
	It holds that $(I-Y_0 Z_0) B_r^{\T} u_1\ge 0$ and 
	$(I-Z_0 Y_0) C_r^{\T} u_2 \ge 0$. 
\end{theorem}

\begin{proof}
	Since 
	\begin{align*}
		W	
		\begin{bmatrix}
			u_1 \\ u_2	
		\end{bmatrix}
		=
		\begin{bmatrix}
			D & -C \\ -B & A
			\end{bmatrix}\begin{bmatrix}
			u_1 \\ u_2 
			\end{bmatrix}=\begin{bmatrix}
			v_1 \\ v_2
		\end{bmatrix},
	\end{align*}
	we then have 
	\begin{align*}
		\begin{bmatrix}
			D_{-\beta} & \alpha C \\ \beta B & A_{-\alpha} 
			\end{bmatrix} \begin{bmatrix}
			u_1 \\ u_2
		\end{bmatrix} =
		\begin{bmatrix}
			D_\alpha  & -\beta C \\ -\alpha B & A_\beta 
			\end{bmatrix} \begin{bmatrix}
			u_1 \\ u_2
			\end{bmatrix} - \gamma \begin{bmatrix}
			v_1 \\ v_2
		\end{bmatrix},
	\end{align*}
	or equivalently  
	\begin{align}
		\gamma D_\alpha^{-1} C u_2 + \gamma D_\alpha^{-1} v_1 
		&=
		u_1 - D_\alpha^{-1} D_{-\beta} u_1, \label{eq:u1-to-u2}
		\\
		\gamma A_\beta^{-1} B u_1 + \gamma A_\beta^{-1} v_2
		&=u_2 -  A_\beta^{-1} A_{-\alpha}  u_2.\label{eq:u2-to-u1}
	\end{align}
	Pre-multiplying $\alpha B_r^{\T}$ and $\beta C_r^{\T}$, respectively,  
	on both sides of \eqref{eq:u1-to-u2} and \eqref{eq:u2-to-u1}, then by \eqref{eq:initial-decoupled-1} we have 
	\begin{align*}
		\gamma Y_0  C_r^{\T} u_2 + \alpha \gamma  Q_0^{\T} v_1 
		&=
		\alpha B_r^{\T} u_1 - \alpha B_r^{\T} D_\alpha^{-1} D_{-\beta} u_1, 
		\\
		\gamma Z_0 B_r^{\T} u_1 + \beta \gamma V_0^{\T}v_2
		&=\beta C_r^{\T} u_2 -\beta C_r^{\T} A_\beta^{-1} A_{-\alpha} u_2. 
	\end{align*}
	These are  further   equivalent to 
	\begin{align}
		\gamma B_r^{\T}u_1 - \gamma Y_0  C_r^{\T} u_2 
		&=
		\beta B_r^{\T} u_1 +  \alpha \gamma  Q_0^{\T} v_1  +  \alpha B_r^{\T} D_\alpha^{-1} D_{-\beta} u_1, \label{eq:u1-to-u2-2}
		\\
		\gamma C_r^{\T}u_2 - \gamma Z_0 B_r^{\T} u_1
		&=\alpha C_r^{\T} u_2 + \beta \gamma V_0^{\T}v_2
		+\beta C_r^{\T} A_\beta^{-1} A_{-\alpha} u_2.\label{eq:u2-to-u1-2}
	\end{align}
	Now we rewrite \eqref{eq:u1-to-u2-2} and \eqref{eq:u2-to-u1-2} as 
	\begin{equation}\label{eq:triple-initial-1}
		\begin{aligned}
			\begin{bmatrix}
				-Y_0 & I \\ I & -Z_0
			\end{bmatrix}
			\begin{bmatrix}
				C_r^{\T}u_2 \\ B_r^{\T} u_1
			\end{bmatrix}
			=& 
			\frac{1}{\gamma}
			\begin{bmatrix}
				\beta B_r^{\T} u_1 +  \alpha \gamma  Q_0^{\T} v_1  +  \alpha B_r^{\T} D_\alpha^{-1} D_{-\beta} u_1
				\\
				\alpha C_r^{\T} u_2 + \beta \gamma V_0^{\T}v_2
				+\beta C_r^{\T} A_\beta^{-1} A_{-\alpha} u_2
			\end{bmatrix}
			\\
			=& 
			\begin{bmatrix}
				\alpha Q_0^{\T} v_1 \\ \beta V_0^{\T} v_2
			\end{bmatrix}
			+
			\begin{bmatrix}
				B_r^{\T} D_\alpha^{-1} u_1
				\\
				C_r^{\T} A_\beta^{-1} u_2
			\end{bmatrix}
			\equiv
			\begin{bmatrix}
				Q_0^{\T}u_1 + \alpha Q_0^{\T} v_1 
				\\ 
				V_0^{\T}u_2 + \beta V_0^{\T} v_2
			\end{bmatrix}
			\ge  0. 
		\end{aligned}
	\end{equation}
	Pre-multiplying  \eqref{eq:triple-initial-1} by $\begin{bmatrix}
		I & Y_0 \\ Z_0 & I
	\end{bmatrix}$,  
	it  shows that 
	\begin{align}\label{eq:triple-initial-2}
		\begin{bmatrix}
			0 & I-Y_0 Z_0 \\ I-Z_0 Y_0 & 0
		\end{bmatrix}
		\begin{bmatrix}
			C_r^{\T} u_2 \\ B_r^{\T} u_1
		\end{bmatrix}
		=
		\begin{bmatrix}
			I & Y_0 \\ Z_0 & I
		\end{bmatrix}
		\begin{bmatrix}
			Q_0^{\T}u_1 + \alpha Q_0^{\T} v_1 
			\\ 
			V_0^{\T}u_2 + \beta V_0^{\T} v_2
		\end{bmatrix}
		\ge  0,
	\end{align}
	implying the results we want to prove.
\end{proof}

With $B_r\ge 0$ and $C_r\ge 0$  of full column rank, 
 we have $B_r^{\T}u_1 > 0$ and $C_r^{\T}u_2 > 0$. 
Hence  from~\eqref{eq:triple-initial-2}, we 
   obtain  the triplet representations of $I-Y_0 Z_0$ 
and $I-Z_0 Y_0$, respectively:  
\begin{equation}\label{eq:triple-initial-3}
	\begin{aligned}
	&
	\left(N_{I-Y_0 Z_0}, B_r^{\T}u_1, Q_0^{\T}u_1 + \alpha Q_0^{\T}v_1  + Y_0 (V_0^{\T}u_2 + \beta V_0^{\T}v_2) \right),
	\\
	&
	\left(N_{I-Z_0Y_0}, C_r^{\T}u_2,  V_0^{\T}u_2 + \beta V_0^{\T}v_2 
	+ Z_0 (Q_0^{\T}u_1+\alpha Q_0^{\T}v_1 )\right). 
\end{aligned}
\end{equation}
Moreover, from the relationships between $Y_0$ 
and $Y_1$, and $Z_0$ and $Z_1$, the triplet representations 
in~\eqref{eq:triple-initial-3} provide further clues for the 
triplet representations of $I-Y_1Z_1$ and $I-Z_1 Y_1$, 
and  those  of $I-Y_kZ_k$ and $I-Z_k Y_k$ for $k>1$. 
Specifically, since 
$\boldsymbol{1}_{2^k} \otimes B_r^{\T} u_1 >0$ and 
$\boldsymbol{1}_{2^k} \otimes C_r^{\T} u_2>0$,
 if 
\begin{equation}\label{eq:condition}
	\begin{aligned}[b]
		&(I-Y_k Z_k) (\boldsymbol{1}_{2^k} \otimes B_r^{\T} u_1)\ge 0, \quad 
		(I-Z_k Y_k) (\boldsymbol{1}_{2^k} \otimes C_r^{\T} u_2)\ge 0,
	\end{aligned}
\end{equation}
we  have   successfully   found the  triplet representations of $I-Y_k Z_k$ and $I-Z_kY_k$. 
To verify  \eqref{eq:condition}, the following theorem is necessary.  

\begin{theorem}\label{thm:nonnegative}
	For the dADDA with $u_1>0$ and $u_2>0$, 	it holds that 
	\begin{enumerate}
		\item [(\romannumeral1)] 
				$[Q_0, Q_1, \cdots,  Q_{2^k-1}]^{\T}
			u_1
			-
			\gamma T_k 
			(\boldsymbol{1}_{2^k} \otimes C_r^{\T} u_2)
			\ge 0$; and \label{item:1} 
		\item [(\romannumeral2)] 
				$[V_0, V_1,  \cdots, V_{2^k-1}]^{\T} 
			u_2
			-
			\gamma S_k 
			(\boldsymbol{1}_{2^k} \otimes B_r^{\T} u_1)
			\ge 0$. \label{item:2}
	\end{enumerate}
\end{theorem}

\begin{proof}
	Pre-multiplying $Q_i^{\T}$ $(0\le i \le 2^k-1,  k\ge 0)$ 
	and $D_{\alpha,\beta}$, respectively, on  
	\eqref{eq:u1-to-u2} yields 
	\begin{align}
		\gamma Q_i^{\T}W_0 C_r^{\T}u_2 + \gamma Q_i^{\T} D_\alpha^{-1} v_1
		& = Q_i^{\T} u_1 - Q_{i+1}^{\T} u_1, \nonumber 
		\\
		\gamma W_1 C_r^{\T}u_2 + \gamma D_{\alpha,\beta} D_\alpha^{-1}v_1
		& = D_{\alpha,\beta}u_1 - D_{\alpha,\beta}^2 u_1. \label{eq:W1}
	\end{align}
Pre-multiply $Q_i^{\T}$ and $D_{\alpha,\beta}$, respectively, to~\eqref{eq:W1} and  we get 
\begin{align}
	\gamma Q_i^{\T}W_1 C_r^{\T}u_2 + \gamma Q_{i+1}^{\T} D_\alpha^{-1}v_1
	& = Q_{i+1}^{\T}u_1 - Q_{i+2}^{\T} u_1, \nonumber 
	\\
	\gamma W_2 C_r^{\T}u_2 + \gamma D_{\alpha,\beta}^2 D_\alpha^{-1}v_1
	& = D_{\alpha,\beta}^2u_1 - D_{\alpha,\beta}^3 u_1. \label{eq:W2}
\end{align}
We pursue the same process as above on \eqref{eq:W2} and obtain 
\begin{align}
	\gamma Q_i^{\T}W_2 C_r^{\T}u_2 + \gamma Q_{i+2}^{\T} D_\alpha^{-1}v_1
	& = Q_{i+2}^{\T}u_1 - Q_{i+3}^{\T} u_1, \nonumber 
	\\
	\gamma W_3 C_r^{\T}u_2 + \gamma D_{\alpha,\beta}^3 D_\alpha^{-1}v_1
	& = D_{\alpha,\beta}^3u_1 - D_{\alpha,\beta}^4 u_1. \label{eq:W3}
\end{align}
Repeating the similar procedure as above, 
we eventually acquire 
\begin{align}
	\gamma Q_i^{\T}W_j C_r^{\T}u_2 + \gamma Q_{i+j}^{\T} D_\alpha^{-1}v_1
	& = Q_{i+j}^{\T}u_1 - Q_{i+j+1}^{\T} u_1, \label{eq:QiWj}
\end{align}
where $j=0, 1, \cdots, 2^k-1$.  
Summing \eqref{eq:QiWj} for  $j=0, 1, \cdots, 2^k-1$,  we  have 
\begin{equation}\label{eq:Qi}
	\begin{aligned}
		&
		\gamma Q_i^{\T}(W_0 + W_1+ \cdots + W_{2^k-1}) C_r^{\T}u_2  
		+ \gamma (Q_i^{\T} + Q_{i+1}^{\T} + \cdots + Q_{i+2^k-1}^{\T}) D_\alpha^{-1}v_1 
		\\
		=&
		Q_i^{\T} u_1 - Q_{i+2^k}^{\T} u_1.  
	\end{aligned}
\end{equation}
Rewrite \eqref{eq:Qi} in matrix form $(0\le i \le 2^k-1)$, we obtain 
\begin{equation*}
	\begin{aligned}[b]
		&
		\gamma T_k
		\begin{bmatrix}
			C_r^{\T} u_2 \\ \vdots  \\ C_r^{\T}u_2
		\end{bmatrix}
		+ \gamma \begin{bmatrix}
			(Q_0^{\T} +\cdots + Q_{2^k-1}^{\T}) D_\alpha^{-1}v_1
			\\
			\vdots \\
			(Q_{2^k-1}^{\T} + \cdots + Q_{2^{k+1}-2}^{\T}) D_\alpha^{-1} v_1
		\end{bmatrix}
		=
		&
		\begin{bmatrix}
			Q_0^{\T}  \\ \vdots \\ Q_{2^k-1}^{\T}
		\end{bmatrix} u_1
		-
		\begin{bmatrix}
			Q_{2^k}^{\T} \\ \vdots \\ Q_{2^{k+1}-1}^{\T} 	
		\end{bmatrix}u_1. 
	\end{aligned}
\end{equation*}
This is equivalent to   
\begin{equation}\label{eq:nonnegativeT}
	\begin{aligned}[b]
		&
		\begin{bmatrix}
			Q_0^{\T}  \\ \vdots \\ Q_{2^k-1}^{\T}
		\end{bmatrix} u_1
		-
		\gamma T_k
		\begin{bmatrix}
			C_r^{\T} u_2 \\ \vdots  \\ C_r^{\T}u_2
		\end{bmatrix}
		=
		\gamma \begin{bmatrix}
			(Q_0^{\T} +\cdots + Q_{2^k-1}^{\T}) D_\alpha^{-1}v_1
			\\
			\vdots \\
			(Q_{2^k-1}^{\T} + \cdots + Q_{2^{k+1}-2}^{\T}) D_\alpha^{-1} v_1
		\end{bmatrix}
		+
		\begin{bmatrix}
			Q_{2^k}^{\T} \\ \vdots \\ Q_{2^{k+1}-1}^{\T} 	
		\end{bmatrix}u_1,
	\end{aligned}
\end{equation}
leading to the result in (\romannumeral1). 
Similarly, we get  
\begin{equation}\label{eq:nonnegativeS}
	\begin{aligned}[b]
		&
		\begin{bmatrix}
			V_0^{\T}  \\ \vdots \\ V_{2^k-1}^{\T}
		\end{bmatrix} u_2
		-
		\gamma S_k
		\begin{bmatrix}
			B_r^{\T} u_1 \\ \vdots  \\ B_r^{\T}u_1
		\end{bmatrix}
		=
		\gamma \begin{bmatrix}
			(V_0^{\T} +\cdots + V_{2^k-1}^{\T}) A_\beta^{-1}v_2
			\\
			\vdots \\
			(V_{2^k-1}^{\T} + \cdots + V_{2^{k+1}-2}^{\T}) A_\beta^{-1} v_2
		\end{bmatrix}
		+
		\begin{bmatrix}
			V_{2^k}^{\T} \\ \vdots \\ V_{2^{k+1}-1}^{\T} 	
		\end{bmatrix}u_2 \ge 0,
	\end{aligned}
\end{equation}
thus the result  in  (\romannumeral2). 
\end{proof}
 
The following part is devoted to the triplet representations of $I-Y_kZ_k$ and $I-Z_k Y_k$, 
for $k\ge 1$. We firstly compute the triplet representations of $I-Y_1Z_1$ and $I-Z_1 Y_1$.  
Define    
\begin{align*}
	P_1 = \begin{bmatrix}
		0 & 0 & I & 0 \\
		I & 0 & 0 &0  \\
		0 & I & 0 & 0 \\
		0 & 0 & 0 & I
	\end{bmatrix} \in \mathbb{R}^{2(p+q)\times 2(p+q)}.
\end{align*}	
Since 
\begin{align*}
	\begin{bmatrix}
		-Y_1 & I \\ I & -Z_1
	\end{bmatrix}
	&= 
	P_1 \begin{bmatrix}
		-\gamma T_0 & 0 & -Y_0 & I 
		\\
		0 & 0 & I & -Z_0 
		\\
		-Y_0 & I & 0 & 0 
		\\
		I & -Z_0 & 0 & -\gamma S_0
	\end{bmatrix}
	P_1^{\T}
	\\
	&=
	P_1\begin{bmatrix}
		0 & 0 & -Y_0 & I 
		\\
		0 & 0 & I & -Z_0 
		\\
		-Y_0 & I & 0 & 0 
		\\
		I & -Z_0 & 0 & 0
	\end{bmatrix}
	P_1^{\T}
	-P_1 \begin{bmatrix}
		I & 0 \\ 0 & 0 \\ 0 & 0 \\ 0 &  I
	\end{bmatrix}
	\begin{bmatrix}
		\gamma T_0  & 0 & 0 & 0 
		\\
		0 & 0 & 0 & \gamma S_0
	\end{bmatrix}
	P_1^{\T},
\end{align*}
then  by \eqref{eq:triple-initial-1} and Theorem~\ref{thm:nonnegative},   it holds that 
\begin{align*}
	&
	\begin{bmatrix}
		-Y_1 & I \\ I & -Z_1
	\end{bmatrix} P_1 
	\begin{bmatrix}
		C_r^{\T} u_2 \\ B_r^{\T} u_1
		\\
		C_r^{\T} u_2 \\ B_r^{\T} u_1
	\end{bmatrix}
	\\
	=&
	P_1 \begin{bmatrix}
		0 & 0 & -Y_0 & I 
		\\
		0 & 0 & I & -Z_0 
		\\
		-Y_0 & I & 0 & 0 
		\\
		I & -Z_0 & 0 & 0
	\end{bmatrix}
	\begin{bmatrix}
		C_r^{\T} u_2 \\ B_r^{\T} u_1
		\\
		C_r^{\T} u_2 \\ B_r^{\T} u_1
	\end{bmatrix}
	- P_1
	\begin{bmatrix}
		I & 0 \\ 0 & 0 \\ 0 & 0 \\ 0 &  I
	\end{bmatrix}
	\begin{bmatrix}
		\gamma T_0 & 0 & 0 & 0
		\\
		0 & 0 & 0 & \gamma S_0
	\end{bmatrix}
	\begin{bmatrix}
		C_r^{\T} u_2 \\ B_r^{\T} u_1
		\\
		C_r^{\T} u_2 \\ B_r^{\T} u_1
	\end{bmatrix}
	\\
	=&
	P_1 \begin{bmatrix}
		Q_0^{\T} u_1 + \alpha Q_0^{\T}v_1  
		- \gamma T_0 C_r^{\T} u_2
		\\
		V_0^{\T}u_2 + \beta V_0^{\T} v_2
		\\
		Q_0^{\T} u_1 + \alpha Q_0^{\T}v_1  
		\\
		V_0^{\T}u_2 + \beta V_0^{\T} v_2
		- \gamma S_0 B_r^{\T} u_1
	\end{bmatrix} \ge 0.
\end{align*}
Moreover,  define 
\begin{align*}
	v_1^{(1)}
	&\equiv 
	\alpha \begin{bmatrix}
		Q_0^{\T} \\ Q_0^{\T} 
		\end{bmatrix}v_1 + \begin{bmatrix}
		Q_0^{\T} \\ Q_1^{\T}
	\end{bmatrix}u_1 
	+\gamma
	\begin{bmatrix}
		0 \\ Q_0^{\T} D_\alpha^{-1}v_1
	\end{bmatrix}\ge 0,
	\\
	v_2^{(1)}
	&\equiv 
	\beta  \begin{bmatrix}
		V_0^{\T} \\ V_0^{\T} 
		\end{bmatrix} v_2 + \begin{bmatrix}
		V_0^{\T} \\ V_1^{\T}
	\end{bmatrix}u_2 
	+\gamma
	\begin{bmatrix}
		0 \\ V_0^{\T} A_\beta^{-1}v_2
	\end{bmatrix}\ge 0,
\end{align*}
it then  follows from  \eqref{eq:nonnegativeT} and \eqref{eq:nonnegativeS}  that 
\begin{equation}\label{eq:triple-repre-1-pre}
	\begin{aligned}
		&\begin{bmatrix}
			-Y_1 & I \\ I & -Z_1 
		\end{bmatrix}
		\begin{bmatrix}
			C_r^{\T} u_2 \\ C_r^{\T} u_2 \\ B_r^{\T} u_1 \\ B_r^{\T} u_1
		\end{bmatrix}
		=
		P_1 
		\begin{bmatrix}
			Q_0^{\T} u_1 + \alpha Q_0^{\T}v_1  
			- \gamma T_0 C_r^{\T} u_2
			\\
			V_0^{\T}u_2 + \beta V_0^{\T} v_2
			\\
			Q_0^{\T} u_1 + \alpha Q_0^{\T}v_1  
			\\
			V_0^{\T}u_2 + \beta V_0^{\T} v_2
			- \gamma S_0 B_r^{\T} u_1
		\end{bmatrix}
		\\
		=&
		\begin{bmatrix}
			Q_0^{\T} u_1 + \alpha Q_0^{\T}v_1
			\\
			\alpha Q_0^{\T} v_1 + \gamma Q_0^{\T} D_\alpha^{-1} v_1 + Q_1^{\T} u_1
			\\
			V_0^{\T}u_2 + \beta V_0^{\T}v_2
			\\
			\beta V_0^{\T}v_2 +  \gamma V_0^{\T} A_\beta^{-1} v_2 + V_1^{\T}u_2 
		\end{bmatrix} 
		\equiv 
		\begin{bmatrix}
			v_1^{(1)} \\ v_2^{(1)}
		\end{bmatrix}
		\ge 0. 
	\end{aligned}
\end{equation}
As a result,   pre-multiplying $\begin{bmatrix}
	I & Y_1 \\ Z_1 & I
\end{bmatrix}$ on both sides leads to
\begin{align*}
	\begin{bmatrix}
		0 & I -Y_1 Z_1 \\  I-Z_1 Y_1 & 0
	\end{bmatrix}
	\begin{bmatrix}
		C_r^{\T}u_2 \\ C_r^{\T}u_2 \\ B_r^{\T}u_1 \\ B_r^{\T}u_1
	\end{bmatrix} 
	=
	\begin{bmatrix}
		I & Y_1 \\ Z_1 & I
	\end{bmatrix}
	\begin{bmatrix}
		v_1^{(1)} \\ v_2^{(1)}
	\end{bmatrix}
	\ge 0,
\end{align*}
which is equivalent to  
\begin{align}\label{eq:triple-repre-1}
	(I-Y_1 Z_1) \begin{bmatrix}
		B_r^{\T}u_1 \\ B_r^{\T}u_1
	\end{bmatrix} = v_1^{(1)} + Y_1 v_2^{(1)}\ge 0, 
	\qquad
	&
	(I-Z_1 Y_1) \begin{bmatrix}
		C_r^{\T}u_2 \\ C_r^{\T}u_2
	\end{bmatrix} = v_2^{(1)} + Z_1 v_1^{(1)}\ge 0. 
\end{align}
Obviously, \eqref{eq:triple-repre-1} yields  the triplet representations of the nonsingular M-matrices  
$I-Y_1 Z_1$ and $I-Z_1 Y_1$, which  respectively are:  
\begin{align*}
	&
	\left(N_{I-Y_1 Z_1},  \boldsymbol{1}_{2} \otimes B_r^{\T}u_1, v_1^{(1)} +Y_1v_2^{(1)}\right), \quad
	\left(N_{I-Z_1 Y_1},  \boldsymbol{1}_{2} \otimes C_r^{\T}u_2, v_2^{(1)} +Z_1v_1^{(1)}\right).
\end{align*}

Following the process presented above, we shall obtain   
the triplet representations  
of  $I- Y_{k}Z_k $ and $I-Z_{k}Y_{k}$ for $k\ge 1$, as presented below.

\begin{theorem}\label{thm:triple-repre}
	Define 
	\begin{align*}
		v_1^{(k)} 
		&= 
		\begin{multlined}[t]
			\alpha (\boldsymbol{1}_{2^k} \otimes Q_0^{\T} v_1) + 
			[Q_0, Q_1, \cdots, Q_{2^k-1}]^{\T} u_1
			\\
			+ \gamma 
			[0, Q_0, Q_0+Q_1, \cdots, Q_0+Q_1+\dots + Q_{2^k-2}]^{\T} D_\alpha^{-1}v_1,
		\end{multlined}	
		\\
		v_2^{(k)} 
		&= 
		\begin{multlined}[t]
			\beta (\boldsymbol{1}_{2^k} \otimes V_0^{\T} v_2) + 
			[V_0, V_1, \cdots, V_{2^k-1}]^{\T} u_2
			\\
			+ \gamma 
			[0, V_0, V_0+V_1, \cdots, V_0+V_1+\dots + V_{2^k-2}]^{\T} A_\beta^{-1}v_2
		\end{multlined}
	\end{align*}
	for $k\ge 1$. Then it holds that 
	\begin{align}\label{eq:triple-repre-k-pre}
		\begin{bmatrix}
			-Y_k & I \\ I & -Z_k
		\end{bmatrix}
		\begin{bmatrix}
			\boldsymbol{1}_{2^k} \otimes C_r^{\T}u_2
			\\
			\boldsymbol{1}_{2^k} \otimes B_r^{\T}u_1
		\end{bmatrix}
		= \begin{bmatrix}
			v_1^{(k)} \\ v_2^{(k)}
		\end{bmatrix}.
	\end{align}
	Moreover, we have 
	\begin{equation}\label{eq:triple-repre-k}
		\begin{aligned}
		(I-Y_k Z_k) (\boldsymbol{1}_{2^k} \otimes B_r^{\T}u_1)
		&= v_1^{(k)} + Y_k v_2^{(k)}, 
		\\
		(I-Z_k Y_k) (\boldsymbol{1}_{2^k} \otimes C_r^{\T}u_2)
		&= v_2^{(k)} + Z_k v_1^{(k)}. 
	\end{aligned}
	\end{equation}
\end{theorem}

\begin{proof}
	We will prove by induction. 
	By \eqref{eq:triple-repre-1-pre} and  
	\eqref{eq:triple-repre-1}, the result is valid  for $k=1$. 
	Now assume that the result holds for $k\ge 2$, then 
	by defining 
	\begin{align*}
		P_2 = \begin{bmatrix}
			0 &0  &I & 0\\
			I & 0 & 0 & 0 \\
			0 & I & 0 & 0 \\
			0& 0& 0& I
		\end{bmatrix} \in \mathbb{R}^{2^{k+1}(p+q)\times 2^{k+1}(p+q)},
	\end{align*}
	it holds that 
	\begin{align*}
		\begin{bmatrix}
			-Y_{k+1} & I \\ I & -Z_{k+1}
		\end{bmatrix}
		=&
		P_2 \begin{bmatrix}
			-\gamma T_k & 0 & -Y_k & I
			\\
			0 & 0 & I & -Z_k 
			\\
			-Y_k & I & 0 & 0
			\\
			I & -Z_k & 0 & -\gamma S_k
		\end{bmatrix} P_2^{\T}
		\\
		=& 
		P_2\begin{bmatrix}
			0 & 0 & -Y_k & I 
			\\
			0 & 0 & I & -Z_k
			\\
			-Y_k & I & 0 & 0 
			\\
			I & -Z_k & 0 & 0
		\end{bmatrix}
		P_2^{\T}
		-P_2 \begin{bmatrix}
			I & 0 \\ 0 & 0 \\ 0 & 0 \\ 0 &  I
		\end{bmatrix}
		\begin{bmatrix}
			\gamma T_k  & 0 & 0 & 0 
			\\
			0 & 0 & 0 & \gamma S_k
		\end{bmatrix}
		P_2^{\T}.
	\end{align*}
	Furthermore, post-multiplying $P_2\begin{bmatrix}
		\boldsymbol{1}_{2^k} \otimes C_r^{\T}u_2
		\\
		\boldsymbol{1}_{2^k} \otimes B_r^{\T}u_1
		\\
		\boldsymbol{1}_{2^k} \otimes C_r^{\T}u_2
		\\
		\boldsymbol{1}_{2^k} \otimes B_r^{\T}u_1
	\end{bmatrix}$  yields  
	\begin{align*}
		&
		\begin{bmatrix}
			-Y_{k+1} & I \\ I & -Z_{k+1}
		\end{bmatrix}
		\begin{bmatrix}
			\boldsymbol{1}_{2^{k+1}} \otimes C_r^{\T} u_2 
			\\
			\boldsymbol{1}_{2^{k+1}} \otimes B_r^{\T} u_1 
		\end{bmatrix}
		\\
		=&
		P_2
		\begin{bmatrix}
			0 & 0 & -Y_k & I 
			\\
			0 & 0 & I & -Z_k 
			\\
			-Y_k & I & 0 & 0 
			\\
			I & -Z_k & 0 & 0
		\end{bmatrix}
		\begin{bmatrix}
			\boldsymbol{1}_{2^k} \otimes C_r^{\T} u_2 
			\\
			\boldsymbol{1}_{2^k} \otimes B_r^{\T} u_1
			\\
			\boldsymbol{1}_{2^k} \otimes C_r^{\T} u_2 
			\\
			\boldsymbol{1}_{2^k} \otimes B_r^{\T} u_1
		\end{bmatrix}
		- P_2 \begin{bmatrix}
			I & 0 \\ 0 & 0 \\ 0 & 0 \\ 0 & I
		\end{bmatrix}
		\begin{bmatrix}
			\gamma T_k & 0 & 0 & 0\\
			0 & 0 & 0& \gamma S_k
		\end{bmatrix}
		\begin{bmatrix}
			\boldsymbol{1}_{2^k} \otimes C_r^{\T} u_2 
			\\
			\boldsymbol{1}_{2^k} \otimes B_r^{\T} u_1
			\\
			\boldsymbol{1}_{2^k} \otimes C_r^{\T} u_2 
			\\
			\boldsymbol{1}_{2^k} \otimes B_r^{\T} u_1
		\end{bmatrix}
		\\
		=&
		P_2\begin{bmatrix}
			v_1^{(k)} - \gamma T_k (\boldsymbol{1}_{2^k} \otimes C_r^{\T} u_2)
			\\
			v_2^{(k)}
			\\
			v_1^{(k)}
			\\
			v_2^{(k)} - \gamma S_k (\boldsymbol{1}_{2^k} \otimes B_r^{\T} u_1)
		\end{bmatrix}
		\equiv  
		\begin{bmatrix}
			v_1^{(k)}
			\\
			v_1^{(k)} - \gamma T_k (\boldsymbol{1}_{2^k} \otimes C_r^{\T} u_2)
			\\
			v_2^{(k)}
			\\
			v_2^{(k)} - \gamma S_k (\boldsymbol{1}_{2^k} \otimes B_r^{\T} u_1)
		\end{bmatrix}. 
	\end{align*}
	Besides, it follows from the definition of $v_1^{(k)}$ and \eqref{eq:nonnegativeT} that 
	\begin{align*}
		&
		v_1^{(k)} - \gamma T_k (\boldsymbol{1}_{2^k} \otimes C_r^{\T} u_2)
		\\
		=& 
		\begin{multlined}[t]
			\alpha (\boldsymbol{1}_{2^k} \otimes Q_0^{\T} v_1) + 
			\begin{bmatrix}
				Q_{2^k}^{\T} 
				\\ 
				Q_{2^k+1}^{\T} 
				\\
				\vdots 
				\\
				Q_{2^{k+1}-1}^{\T}  
			\end{bmatrix} u_1
			+ \gamma 
			\begin{bmatrix}
				Q_0^{\T} + Q_1^{\T}+\cdots + Q_{2^k-1}^{\T}  
				\\
				Q_0^{\T} + Q_1^{\T} + \cdots + Q_{2^k}^{\T}   
				\\
				\vdots 
				\\
				Q_0^{\T}+Q_1^{\T}+\cdots + Q_{2^{k+1}-2}^{\T}
			\end{bmatrix} D_\alpha^{-1}v_1
		\end{multlined},	
	\end{align*}
	indicating that 
	\begin{align*}
		&
		\begin{bmatrix}
			v_1^{(k)} 
			\\ 
			v_1^{(k)} - \gamma T_k (\boldsymbol{1}_{2^k} \otimes C_r^{\T} u_2)
		\end{bmatrix}
		\\
		=&
		\begin{multlined}[t]
			\begin{bmatrix}
				\alpha (\boldsymbol{1}_{2^k} \otimes Q_0^{\T} v_1)
				\\
				\alpha (\boldsymbol{1}_{2^k} \otimes Q_0^{\T} v_1)
			\end{bmatrix}
			+
				[Q_0, Q_1, \cdots, Q_{2^k-1}, Q_{2^k}, Q_{2^k+1}, \cdots, Q_{2^{k+1}-1} ]^{\T} u_1
			\\
			+ \gamma  
				[0 , Q_0  , \cdots , Q_0+Q_1+\dots + Q_{2^k-1}
				, \cdots , 
				Q_0+Q_1+\dots + Q_{2^{k+1}-2}]^{\T} D_\alpha^{-1}v_1
		\end{multlined}
		\\
		\equiv & v_1^{(k+1)}.
	\end{align*}
	Analogously, with the definition of $v_2^{(k)}$ and \eqref{eq:nonnegativeS}, we have 
	\begin{align*}
		\begin{bmatrix}
			v_2^{(k)} 
			\\ 
			v_2^{(k)} - \gamma S_k (\boldsymbol{1}_{2^k} \otimes B_r^{\T} u_1)
		\end{bmatrix} \equiv v_2^{(k+1)}.
	\end{align*}
	Consequently, we obtain  
	\begin{align*}
		\begin{bmatrix}
			-Y_{k+1} & I \\ I & -Z_{k+1}
		\end{bmatrix}
		\begin{bmatrix}
			\boldsymbol{1}_{2^{k+1}} \otimes C_r^{\T} u_2 
			\\
			\boldsymbol{1}_{2^{k+1}} \otimes B_r^{\T} u_1 
		\end{bmatrix}
		= 
		\begin{bmatrix}
			v_1^{(k+1)} \\ v_2^{(k+1)}
		\end{bmatrix}.
	\end{align*}
The proof for \eqref{eq:triple-repre-k-pre} by induction is complete.  
Pre-multiplying $\begin{bmatrix}
	I & Y_k \\ Z_k & I
\end{bmatrix}$ to  \eqref{eq:triple-repre-k-pre}  leads to  \eqref{eq:triple-repre-k}. 
\end{proof}

Obviously, $v_1^{(k)}\ge 0$ and $v_2^{(k)}\ge 0$. So~\eqref{eq:triple-repre-k} 
gives the triplet representations of the nonsingular M-matrices $I-Y_k Z_k$ and 
$I-Z_k Y_k$, which are respectively  
\begin{align}
	&
	\left(
	N_{I-Y_k Z_k}, \boldsymbol{1}_{2^k} \otimes B_r^{\T}u_1, v_1^{(k)}+Y_kv_2^{(k)}
	\right), \quad 
	\left(
	N_{I-Z_k Y_k}, \boldsymbol{1}_{2^k} \otimes C_r^{\T}u_2, v_2^{(k)}+Z_kv_1^{(k)}
	\right).\label{eq:tri-rep-ZYk} 
\end{align}

\subsection{Algorithm for dADDA}\label{ssec:algorithm}

With the triplet representations of $D_\alpha$, $A_\beta$ and $I-Y_k Z_k$, 
 we can  compute $H_k$ in highly accurate using 
the GTH-like algorithm. Our proposed algorithm is summarized as  
Algorithm~\ref{alg:decoupled-highly-accurate-adda}. 

\begin{algorithm}[h]
	\caption{dADDA}\label{alg:decoupled-highly-accurate-adda}
	\hspace*{0.02in}{\bf Input:}
	coefficients $A, D, B_l, B_r, C_l, C_r$ and vectors $u_1, u_2, v_1, v_2$. 
	\\
	\hspace*{0.02in}{\bf Output:}
	the minimal nonnegative solution $X$. 
	\begin{algorithmic}[1]
		\State choose $\alpha$, $\beta$ with $0\le \alpha \le \min_i a_{ii}^{-1}$,  $0 \le \beta\le \min_j d_{jj}^{-1}$, $\max\{\alpha, \beta\}> 0$; 
		\State construct the triplet representations of $D_\alpha$ and $A_\beta$ by \eqref{eq:triple-repreAD}; 
		\State compute $U_0, V_0, Q_0$ and $W_0$ by the GTH-like algorithm, using the triplet representations of $D_\alpha$ and $A_\beta$  in \eqref{eq:triple-repreAD1};    
		\State compute  $Y_0=\alpha Q_0^{\T}C_l$ and $Z_0=\beta C_r^{\T} U_0$;
		\State compute $H_0$ by the GTH-like algorithm, using the triplet representation in \eqref{eq:triple-initial-3}; 
		\State set $k=0$;
		\Repeat 
		\State compute $T_k$ and $S_k$;
		\State set k=k+1;
		\State compute $Y_k$ and $Z_k$ by~\eqref{eq:YZ};
		\State compute $U_j, V_j, Q_j$ and $W_j$ by the GTH-like algorithm for $j=2^{k-1}, \cdots, 2^k-1$; 
		\State compute $v_1^{(k)}$ and $v_2^{(k)}$ by Theorem~\ref{thm:triple-repre};
		\State compute $H_k$ by the GTH-like algorithm, using triplet representation  in \eqref{eq:tri-rep-ZYk};  
		\Until{convergence}
		\\
		\Return the last $H_k$ as the approximation to $X$. 
	\end{algorithmic}
\end{algorithm}

For the implementation of Algorithm~\ref{alg:decoupled-highly-accurate-adda}, 
we need to choose the 
stop criteria for convergence. With $\varepsilon$ being a  preselected tolerance, 
 we may adopt one of the following criteria: 
\begin{enumerate}
	\item[(1)] The normalized residual in norm: 
\[
	\frac{\|H_k C H_k - H_k D -AH_k+B\|}
{\|H_k C H_k\|+\|H_k D\|+\|A H_k\| + \|B\|}\le \varepsilon,
\]
where $\|\cdot\|$ is some matrix norm and for convenience one can use the Frobenius norm or 
the $l_1$ norm.

\item[(2)] The relative change: 
	\[
		|H_{k+1}-H_k| \le \varepsilon H_{k+1} ,
\]
which is simple and cheap to use.

\item[(3)] The entrywise relative residual: 
\[
	\mathrm{ERres}_k:=\max_{i,j}
	\frac{|(H_kCH_k +  N_{A} H_k + H_kN_D + B) - (\diag(A) H_k + H_k \diag(D))|_{(i,j)}}
	{[\diag(A) H_k + H_k \diag(D)]_{(i,j)}}\le \varepsilon,
\]
which is 
the entrywise relative 
accuracy of $H_k$ as an approximation to $X$.  

\item[(4)] The entrywise relative error: 
\[
	\mathrm{ERerr}_k:= \max_{i,j} \frac{|(H_k-X)_{(i,j)}|}{X_{(i,j)}} \le \varepsilon,
\]
which is not generally available because $X$ is unknown. 
\end{enumerate}

As we are interested in the accuracy of the entries in $X$, thus 
for Algorithm~\ref{alg:decoupled-highly-accurate-adda} we recommend   
the entrywise relative residual ERres$_k$ in the convergence control. 
In \cite{wangWL2012alternatingdirectional, xueXL2012accurate},  
the Kahan's stopping criteria was recommended for terminating iterations, which may  
lead premature termination without improvements in the approximate solution; please refer 
to \cite{xueXL2012accurateSylvester} for details.

\begin{remark}\label{rk:complexitydADDA}
	The  dominant computational cost for  Algorithm~\ref{alg:decoupled-highly-accurate-adda} involves $U_j, V_j, Q_j, W_j$ and $H_k$, which are obtained by solving M-matrix linear systems. More concretely,  it follows from 
	\begin{align*}
		&U_0 = A_\beta^{-1}B_l,  \quad V_0 = A_\beta^{-\T}C_r, \quad  
		W_0 = D_\alpha^{-1} C_l, \quad Q_0 =D_\alpha^{-\T} B_r,
		\\
		&U_j=A_{\alpha,\beta} U_{j-1} = A_{-\alpha}A_{\beta}^{-1}U_{j-1},
		\quad
		V_j=A_{\alpha,\beta}^{\T} V_{j-1}=A_{-\alpha}^{\T}A_{\beta}^{-\T}V_{j-1},
		\\
		&W_j=D_{\alpha,\beta} W_{j-1}=D_{-\beta}D_{\alpha}^{-1}W_{j-1},
		\quad
		Q_j= D_{\alpha,\beta}^{\T} Q_{j-1}=D_{-\beta}^{\T}D_{\alpha}^{-\T}Q_{j-1}    
	\end{align*}
	that it requires $\bigO(m^2(p+q) + n^2(p+q))$ flops as long as the LU factorizations of the M-matrices $A_{\beta}$ and $D_{\alpha}$ are known.  For $H_k$ we just need to solve the small M-matrix linear systems with the coefficient being $I-Y_k Z_k\in \mathbb{R}^{2^k p \times 2^k p}$, which involves $\bigO((2^kp)^3)$ flops . Accordingly, in each iteration the computational complexity  is $\bigO(m^2+n^2)$ and the LU factorizations of  $A_{\beta}$ and $D_{\alpha}$, obtained by performing Algorithm~\ref{alg:GTH-like}, 	dominate the whole  computational cost, which is $\bigO(m^3+n^3)$.

However,  by Algorithm~\ref{alg:GTH-like},  when the original $A$  and $D$ are structurally  sparse or low-rank updates of some diagonal matrices, the complexities for calculating these LU factorizations can be significantly reduced. More specifically, when $A$ and $D$ are banded or   low-rank updates of some diagonal matrices, the complexities  for the  LU factorizations of $A_{\beta}$ and $D_{\alpha}$, respectively,  are $\bigO(m)$ and $\bigO(n)$, and that  for computing $U_j$, $V_j$, $W_j$ and $Q_j$ is $\bigO(m)+\bigO(n)$.  As a result, the total complexity for Algorithm~\ref{alg:decoupled-highly-accurate-adda} is reduced to $\bigO(m+n)$.

	It is worthwhile to point that  in the accADDA proposed by Xue and Li~\cite{xueL2017highly}, one cannot take  advantage of special structures  in  $A$ and $D$, such as sparsity,  because the doubling iterates~\eqref{eq:sda} destroy these structures without modifications. In each iteration the accADDA requires solving two M-matrix linear systems whose sizes are $m\times m$ and $n\times n$, and  the complexity is $\bigO(m^3+n^3)$ per iteration.
\end{remark}

\begin{remark}\label{rk:algorithmYk}
One may also be interested in the dual solution $Y$ to the dual  
 MARE~\eqref{eq:dual_nare}. In that case it is necessary  
to compute  the triplet representation of the nonsingular M-matrix 
$I-Z_kY_k$ in~\eqref{eq:tri-rep-ZYk}. 
We can then  compute $G_k$ in high accuracy, 
adapting the GTH-like algorithm. 
\end{remark}

\section{Numerical Examples}\label{sec:numerical-examples}

To illustrate the performance of the  dADDA, we apply it to two  
test sets. One comes from stochastic fluid flows~\cite{beanOT2005algorithm,xueL2017highly} with $10$ examples. 
The  other originates from the transport theory, 
 the  one-group neutron transport equation~\cite{juangL1998nonsymmetric}. 
 For comparison we  also apply  the accADDA~\cite{xueL2017highly} to both test sets. For all examples $A$ and $D$ are rank-one updates of diagonal matrices and by Remark~\ref{rk:GTH-for-lowrank} one can solve the linear equations related to $A_{\beta}$ and $D_{\alpha}$ with $\bigO(n+m)$ flops.  We replace the GTH-like algorithm  in the  third and eleventh lines in Algorithm~\ref{alg:decoupled-highly-accurate-adda}   with the technique given in Remark~\ref{rk:GTH-for-lowrank} to  compute $U_j$, $V_j$, $W_j$ and $Q_j$, and denote the refined method as dADDA$\rm{_{opt}}$.       
 All three algorithms are implemented in MATLAB 2019b on a 64-bit PC with an Intel Core i7 processor at 3.20 GHz and 64G RAM. 

\begin{example}[Stochastic fluid flow]\label{eg:markov-chain}
	In this example, we have 
	\begin{align*}
		A = n I_{m}, \quad D = (10^4 n + m) I_n - 10^4 \boldsymbol{1}_{n\times n}, \quad B_l = \boldsymbol{1}_{m}, \quad B_r=\boldsymbol{1}_{n}, \quad C_l=\boldsymbol{1}_{n}, \quad C_r=\boldsymbol{1}_{m},
	\end{align*}
	which satisfies $W \boldsymbol{1}_{n+m} = 0$, indicating $u_1=\boldsymbol{1}_{n}$, $u_2=\boldsymbol{1}_{m}$, $v_1=0$ and $v_2=0$. 
	The minimal nonnegative solution  is $X = \frac{1}{n}\boldsymbol{1}_{m\times n}$. When taking $m=2, n=18$, it is exactly the example of a positive recurrent Markov chain  displayed in~{\cite[Example~6.1]{xueL2017highly}}.  
	We set the tolerance for the entrywise relative residual ERres$_k$ as $10^{-14}$ and the maximal number of iterations as $20$.  

	Table~\ref{table-Markov} shows the numerical results produced by accADDA, dADDA and dADDA$\rm{_{opt}}$. 
	Besides ERres$_k$ and ERerr$_k$, we also present  rank$(H_k)$, $\|H_k\|_F$, the   numbers of iterations (\#it)  required  and  the respective execution times (eTime).  
From Table~\ref{table-Markov}, with the same iterations, all three algorithms produce  comparable results on   ERres$_k$,  
ERerr$_k$,  $\rank(H_k)$ and $\|H_k\|_F$.  However, for  examples of medium and large sizes the dADDA requires much less    
execution time  than that of accADDA;   
and the modified method dADDA$\rm{_{opt}}$ takes the least time.  

\begin{table}[t]
	\footnotesize
	\centering
	\begin{tabular}{c|c|c|c|c|c|c}
		\hline 
		\multicolumn{7}{c}{$m=18,n=2$}  \\
		\hline
		&ERres$_k$ & ERerr$_k$& $\rank(H_k)$& $\|H_k\|_F$ & \#it  & eTime    \\
		& ($\times10^{-16}$) & ($\times10^{-12}$) &      & ($\times10^{-1}$) &  & (s)     \\
		\hline 
		accADDA &$5.7773$ &  $5.8178$ & $1$ &  $3.3333$ & $4$ & $3.7296\times 10^{-3}$\\
		dADDA &$3.8515$ & $5.8175$ & $1$& $3.3333$ & $4$& $5.6041\times 10^{-3}$ \\
		dADDA$\rm{_{opt}}$ & $3.8515$& $5.8181$&$1$&  $3.3333$ & $4$ & $ 2.3085\times 10^{-3}$ \\
		\hline 
		\hline
		\multicolumn{7}{c}{$m=90,n=10$}  \\
		\hline
		&ERres$_k$ & ERerr$_k$& $\rank(H_k)$& $\|H_k\|_F$ & \#it  & eTime    \\
		& ($\times10^{-15}$) & ($\times10^{-12}$) &      & ($\times10^{-1}$) &  & (s)     \\
		\hline 
		accADDA & $1.8392$ & $5.8188$ & $1$ & $3.3333$ & $4$ & $4.3765\times 10^{-2}$ \\
		dADDA &   $2.5749$ & $5.8186$ &  $1$  &  $3.3333$ &  $4$ &  $3.0089\times 10^{-2}$ \\
	dADDA$\rm{_{opt}}$ & $0.7357$ & $5.8186$ & $1$ &  $3.3333$  & $4$ &  $2.9446\times 10^{-3}$ \\
		\hline 
		\hline
		\multicolumn{7}{c}{$m=180,n=20$}  \\
		\hline
		&ERres$_k$ & ERerr$_k$& $\rank(H_k)$& $\|H_k\|_F$ & \#it  & eTime    \\
		& ($\times10^{-15}$) & ($\times 10^{-12}$) &      & ($\times10^{-1}$) &  & (s)     \\
		\hline 
		accADDA & $0.7316$ &  $5.8203$ & $1$ & $3.3333$ & $4$ & $1.4570\times 10^{-1}$\\
		dADDA & $5.1211$ &$5.8185$ & $1$ & $3.3333$& $4$ & $5.3418\times 10^{-2}$ \\
		dADDA$\rm{_{opt}}$ &  $1.2803$ & $5.8183$ &$1$ & $3.3333$  &$4$ & $1.9806\times 10^{-3}$ \\
		\hline 
		\hline
		\multicolumn{7}{c}{$m=540,n=60$}  \\
		\hline
		&ERres$_k$ & ERerr$_k$& $\rank(H_k)$& $\|H_k\|_F$ & \#it  & eTime    \\
		& ($\times10^{-15}$) & ($\times10^{-12}$) &      & ($\times10^{-1}$) &  & (s)     \\
		\hline 
		accADDA & $3.6443$  & $5.8197$ &  $1$ &  $3.3333$ & $4$ & $5.5278\times 10^{0}$\\ 
		dADDA &  $2.7332$ &  $5.8183$ & $1$  & $3.3333$  & $4$ & $6.9871\times 10^{-1}$ \\
		dADDA$\rm{_{opt}}$ & $2.1866$ & $5.8173$ & $1$ &  $3.3333$  & $4$  & $1.6774\times 10^{-2}$ \\
		\hline 
		\hline
		\multicolumn{7}{c}{$m=900,n=100$}  \\
		\hline
		&ERres$_k$ & ERerr$_k$& $\rank(H_k)$& $\|H_k\|_F$ & \#it  & eTime    \\
		& ($\times10^{-15}$) & ($\times10^{-12}$) &      & ($\times10^{-1}$) &  & (s)     \\
		\hline 
		accADDA & $1.2746$ & $5.8153$  &  $1$ & $3.3333$ & $4$ & $2.9089\times 10^{1}$\\
		dADDA &  $2.9133$ &  $5.8180$ & $1$ &  $3.3333$  & $4$ & $3.1702\times 10^{0}$ \\
		dADDA$\rm{_{opt}}$ & $2.0029$  & $5.8174$  &  $1$ &  $3.3333$ & $4$ & $1.8585\times 10^{-2}$ \\
		\hline 
		\hline
		\multicolumn{7}{c}{$m=1800,n=200$}  \\
		\hline
		&ERres$_k$ & ERerr$_k$& $\rank(H_k)$& $\|H_k\|_F$ & \#it  & eTime    \\
		& ($\times10^{-15}$) & ($\times10^{-12}$) &      & ($\times10^{-1}$) &  & (s)     \\
		\hline 
		accADDA &  $3.8216$ & $5.8192$&   $1$&  $3.3333$ &  $4$  &$2.1508\times 10^{2}$\\
		dADDA & $3.2756$ & $5.8203$ &   $1$ &   $3.3333$ & $4$ & $2.2739\times 10^{1} $ \\
		dADDA$\rm{_{opt}}$ & $1.8198$  & $5.8168$ &  $1$ &   $3.3333$ & $4$ &  $1.1996\times 10^{-1}$ \\
		\hline 
		\hline
		\multicolumn{7}{c}{$m=3600,n=400$}  \\
		\hline
		&ERres$_k$ & ERerr$_k$& $\rank(H_k)$& $\|H_k\|_F$ & \#it  & eTime    \\
		& ($\times10^{-15}$) & ($\times10^{-12}$) &      & ($\times10^{-1}$) &  & (s)     \\
		\hline 
		accADDA &  $3.2747$ &  $5.8186$ &  $1$ &  $3.3333$ &   $4$ & $1.8254\times 10^{3}$\\
		dADDA &  $3.8205$ &  $5.8153$ &$1$ & $3.3333$ & $4$  & $1.8411\times 10^{2}$\\
		dADDA$\rm{_{opt}}$ & $4.7302$ & $5.8244$  & $1$ &   $3.3333$  &  $4$  & $5.1599\times 10^{-1}$ \\
		\hline 
		\hline
		\multicolumn{7}{c}{$m=7200,n=800$}  \\
		\hline
		&ERres$_k$ & ERerr$_k$& $\rank(H_k)$& $\|H_k\|_F$ & \#it  & eTime    \\
		& ($\times10^{-15}$) & ($\times10^{-12}$) &      & ($\times10^{-1}$) &  & (s)     \\
		\hline 
		accADDA &  $4.7295$ &$5.8434$ &  $1$ &  $3.3333$ & $4$ &   $1.4119\times 10^{4}$ \\
		dADDA & $4.7295$ &  $5.8182$ &  $1$ &  $3.3333$ &  $4$ &  $1.4107\times 10^{3}$ \\
		dADDA$\rm{_{opt}}$ & $1.2733$ & $5.8120$ &   $1$  &$3.3333$ &  $4$ & $2.9070 \times 10^{0}$\\
		\hline 
		\hline
		\multicolumn{7}{c}{$m=10800,n=1200$}  \\
		\hline
		&ERres$_k$ & ERerr$_k$& $\rank(H_k)$& $\|H_k\|_F$ & \#it  & eTime    \\
		& ($\times10^{-15}$) & ($\times10^{-12}$) &      & ($\times10^{-1}$) &  & (s)     \\
		\hline 
		accADDA &   $2.7284$ & $5.7754$  &  $1$ & $3.3333$ & $4$   & $4.8283\times 10^{4}$\\
		dADDA &  $8.0034$ & $5.8238$ &  $1$  & $3.3333$  &  $4$ &  $4.9531\times 10^{3}$\\
		dADDA$\rm{_{opt}}$ &  $3.6379$ & $5.8172$ &  $1$   &  $3.3333$ & $4$ &$1.0210\times10^{1}$\\
		\hline 
		\hline
		\multicolumn{7}{c}{$m=13500,n=1500$}  \\
		\hline
		&ERres$_k$ & ERerr$_k$& $\rank(H_k)$& $\|H_k\|_F$ & \#it  & eTime    \\
		& ($\times10^{-15}$) & ($\times10^{-12}$) &      & ($\times10^{-1}$) &  & (s)     \\
		\hline 
		accADDA &  $3.6378$ & $5.8552$ &  $1$  &  $3.3333$ &   $4$ & $9.4267\times 10^{4}$\\
		dADDA &  $8.0033$ &  $5.8251$ &  $1$ &  $3.3333$ &  $4$ & $9.6697\times 10^{3}$\\
		dADDA$\rm{_{opt}}$ & $0.1819$  &  $5.8166$ &  $1$ & $3.3333$ & $4$ &  $1.9306\times 10^{1}$\\
		\hline 
	\end{tabular}
	\caption{Numerical results for Example~\ref{eg:markov-chain}}
	\label{table-Markov}
\end{table}
\end{example}

\begin{example}[Transport Theory]\label{eg:transport-theory}
	When using the Gauss-Legendre to discretize  the integrodifferential equation satisfied by the scattering function \cite{juangL1998nonsymmetric}, it leads to the MAREs with 
	\begin{align*}
		&
		A = \frac{1}{\beta(1+\alpha)}\diag(\omega_1^{-1}, \cdots, \omega_n^{-1}) - \boldsymbol{1}_n q^{\T}, \quad 
		D =  \frac{1}{\beta(1-\alpha)}\diag(\omega_1^{-1}, \cdots, \omega_n^{-1}) - q\boldsymbol{1}_n^{\T}, \quad 
		\\
		& 
		B_l = B_r = \boldsymbol{1}_n, \quad C_l = C_r = q, 
	\end{align*}
	where $\omega_1, \cdots, \omega_n$ are the Gauss-Legendre notes satisfying $0<\omega_n<\omega_{n-1}<\cdots<\omega_1<1$,  $q = \frac{1}{2} \diag(\omega_1^{-1}, \cdots, \omega_n^{-1}) c$,  and $c$ is the weights vector  with $\sum_{i=1}^n c_i =1$ and $c_i>0$. 

	In this test set, we will randomly generate $\alpha$, $\beta$, $\omega_i$ and $c_i$ as follows: $\alpha$, $\beta$ and $\omega_i$ follow the uniform distribution in the interval $(0,1)$, and we obtain $c$ by the command   \verb|randn| in MATLAB   before the normalization   $\sum_{i=1}^n c_i =1$. We have $W=\begin{bmatrix}
		D & -C \\ -B & A
		\end{bmatrix}$ being a rank-one update of a nonsingular diagonal matrix. Its inverse   can be computed cheaply with the help of the SMWF. 
		For the triplet representation of $W$,  we firstly  generate $v_1\ge 0$ and $v_2\ge 0$ with the command  \verb|rand|, and then we compute $\begin{bmatrix}
		u_1\\u_1
		\end{bmatrix} = W^{-1}\begin{bmatrix}
		v_1 \\ v_2
	\end{bmatrix}$, which obviously holds $u_1>0$ and $u_2>0$. 
	Again, we  take   REres$_k$  as the stop criterion,  
	with $\varepsilon$ setting as $10^{-13}$,  
	and  the maximal number of iterations as $100$. 
Table~\ref{table-Trans} displays  the numerical results produced by the  three algorithms  for ten examples.  
	Within $100$ iterations, the accADDA does not converge, while the dADDA and the dADDA$\rm{_{opt}}$ produce satisfactory results within  $10$ iterations. The dADDA$\rm{_{opt}}$ take much less time than dADDA, 
taking  advantage of the structures in $A$ and $D$.

	\begin{table}[ht!]
		\footnotesize
		\centering
		\begin{tabular}{c|c|c|c|c|c}
			\hline 
			\multicolumn{6}{c}{$n=10$, $\alpha=9.7059\times 10^{-1}$, $\beta=1.5761\times 10^{-1}$}  \\
			\hline
			&ERres$_k$ & $\rank(H_k)$& $\|H_k\|_F$ & \#it  & eTime(s)   \\
			\hline 
			accADDA & $2.9147\times10^{-6}$ &$7$ & $3.1800\times10^{-2}$&$100$&$4.6493\times10^{-2}$  \\
			dADDA & $1.3189\times10^{-15}$  & $6$ & $3.1800\times10^{-2}$&$4$ & $4.5794\times10^{-3}$\\
			dADDA$\rm{_{opt}}$ & $1.5393\times 10^{-15}$ & $6$ & $3.1800\times 10^{-2}$ & $4$ & $2.6813\times 10^{-3}$ \\
			\hline 
			\hline
			\multicolumn{6}{c}{$n=20$, $\alpha=2.5428 \times 10^{-1}$, $\beta= 8.4072\times 10^{-1}$}  \\
			\hline
			&ERres$_k$ & $\rank(H_k)$& $\|H_k\|_F$ & \#it  & eTime(s)    \\
			\hline 
			accADDA &  $1.6249\times 10^{-2}$ & $14$ & $7.4264$ & $100$  & $9.8583\times 10^{-2}$\\ 
			dADDA & $2.3140\times 10^{-15}$ &  $12$ &   $7.5668$ & $7$ & $7.2822\times 10^{-2}$\\
			dADDA$\rm{_{opt}}$ &  $1.3061\times 10^{-15}$ &  $12$  &  $7.5668$ &  $7$  & $1.4355\times 10^{-2}$  \\
			\hline 
			\hline
			\multicolumn{6}{c}{$n=40$, $\alpha= 5.4972\times 10^{-1}$, $\beta=5.8526\times 10^{-1}$}  \\
			\hline
			&ERres$_k$ & $\rank(H_k)$& $\|H_k\|_F$ & \#it  & eTime(s)    \\
			\hline 
			accADDA & $9.1830\times10^{-3}$& $14$ & $5.6640$ & $100$ & $2.4501\times10^{-1}$ \\
			dADDA & $1.8599\times10^{-15}$ & $11$ & $5.7022$ &$6$ & $6.5154\times 10^{-2}$  \\
			dADDA$\rm{_{opt}}$ &$1.5379\times 10^{-15}$ & $11$  & $5.7022$ & $6$ & $6.2749\times 10^{-3}$\\
			\hline 
			\hline
			\multicolumn{6}{c}{$n=50$, $\alpha=1.8896\times 10^{-1}$, $\beta=7.4469\times 10^{-1}$}  \\
			\hline
			&ERres$_k$ & $\rank(H_k)$& $\|H_k\|_F$ & \#it  & eTime(s)    \\
			\hline 
			accADDA & $4.0316\times 10^{-2}$ &  $21$ &  $1.3570\times 10^{1}$  &  $100$ &  $4.2209\times 10^{-1}$\\
			dADDA & $1.5520\times 10^{-14}$ & $16$ &  $1.4245\times 10^{1}$ &   $10$ &  $1.0664\times 10^{1}$ \\
			dADDA$\rm{_{opt}}$ & $6.0738\times 10^{-15}$ & $16$ & $1.4245\times 10^{1}$  &  $10$  &$3.8279\times 10^{0}$ \\
			\hline 
			\hline
			\multicolumn{6}{c}{$n=75$, $\alpha=3.0541\times 10^{-2}$, $\beta=1.9781\times 10^{-1}$}  \\
			\hline
			&ERres$_k$ & $\rank(H_k)$& $\|H_k\|_F$ & \#it  & eTime(s)    \\
			\hline 
			accADDA & $5.6262\times10^{-2}$ & $19$ & $3.8084$ & $100$ & $8.6167\times10^{-1}$ \\
			dADDA & $9.8718\times10^{-15}$ & $14$ & $3.9519$ & $9$ & $4.7539\times 10^{0}$  \\
			dADDA$\rm{_{opt}}$ &  $3.8516\times 10^{-15}$ & $14$ & $3.9519$ & $9$ & $4.9843\times 10^{-1}$	\\
			\hline 
			\hline
			\multicolumn{6}{c}{$n=100$, $\alpha=6.7612\times 10^{-1}$, $\beta=2.4071\times 10^{-1}$}  \\
			\hline
			&ERres$_k$ & $\rank(H_k)$& $\|H_k\|_F$ & \#it  & eTime(s)    \\
			\hline 
			accADDA & $3.2810\times 10^{-2}$ & $19$ & $3.4759$ & $100$  & $1.4790\times 10^{0}$ \\
			dADDA & $1.2228\times 10^{-14}$  &  $14$ &  $3.5386$ &$8$ & $1.9340\times 10^{0}$ \\
			dADDA$\rm{_{opt}}$ & $3.4579\times 10^{-15}$ & $14$ & $3.5386$ & $8$ & $9.3479\times 10^{-2}$ \\
			\hline 
			\hline
			\multicolumn{6}{c}{$n=200$, $\alpha=2.5104\times 10^{-1}$, $\beta=2.1756\times 10^{-1}$}  \\
			\hline
			&ERres$_k$ & $\rank(H_k)$& $\|H_k\|_F$ & \#it  & eTime(s)    \\
			\hline 
			accADDA &   $6.7672\times10^{-2}$ & $24$ & $9.3653$ & $100$ & $8.7763\times 10^{0}$  \\
			dADDA &  $8.7390\times10^{-14}$ & $18$ &$9.8271$ & $10$ & $8.3929\times10^{1}$ \\
			dADDA$\rm{_{opt}}$ &$7.8890\times 10^{-14}$ & $18$ & $9.8271$ & $10$ & $4.4624\times 10^{0}$\\
			\hline 
			\hline
			\multicolumn{6}{c}{$n=400$, $\alpha=9.8284\times 10^{-1}$, $\beta=4.0239\times 10^{-1}$}  \\
			\hline
			&ERres$_k$ & $\rank(H_k)$& $\|H_k\|_F$ & \#it  & eTime(s)    \\
			\hline 
			accADDA & $2.6961\times10^{-2}$ & $14$ &$1.5130$ & $100$ & $7.9537\times10^{1}$   \\
			dADDA & $7.5368\times10^{-14}$ & $10$ & $1.5155$ & $10$ & $4.6554\times10^{2}$ \\
			dADDA$\rm{_{opt}}$ &  $8.4919\times 10^{-15}$ &  $10$ & $1.5155$ &  $10$ & $8.4121\times 10^{0}$\\
			\hline 
			\hline
			\multicolumn{6}{c}{$n=600$, $\alpha=9.5635\times 10^{-1}$, $\beta=6.4509\times 10^{-1}$}  \\
			\hline
			&ERres$_k$ & $\rank(H_k)$& $\|H_k\|_F$ & \#it  & eTime(s)    \\
			\hline 
			accADDA & $7.6493\times10^{-3}$ & $14$ & $9.9458$& $100$& $3.4362\times10^2$ \\
			dADDA & $3.3671\times10^{-14}$ & $12$ & $9.9643$ & $7$ & $1.3553\times10^1$ \\
			dADDA$\rm{_{opt}}$ &$2.0565\times 10^{-14}$ & $12$ &  $9.9643$ & $7$ &  $1.8830\times 10^{-1}$			\\
			\hline 
			\hline
			\multicolumn{6}{c}{$n=800$, $\alpha=9.6990\times 10^{-1}$, $\beta=7.1803\times 10^{-1}$}  \\
			\hline
			&ERres$_k$ & $\rank(H_k)$& $\|H_k\|_F$ & \#it  & eTime(s)    \\
			\hline 
			accADDA & $1.2776\times10^{-2}$ &$15$ & $1.0030\times10^1$ &$100$ & $8.5507\times10^2$   \\
			dADDA & $5.4891\times10^{-14}$ & $11$ & $1.0049\times10^1$& $9$& $4.3877\times10^2$ \\
			dADDA$\rm{_{opt}}$ &  $6.2871\times 10^{-15}$ &  $11$ & $1.0049\times 10^{1}$ &  $9$ & $3.0791\times 10^{0}$ 		\\
			\hline 
		\end{tabular}
		\caption{Numerical results for Example~\ref{eg:transport-theory}}
		\label{table-Trans}
	\end{table}

	For those test examples we observe that different $\alpha$ and $\beta$ would 
	make an impact on the number of iterations  required by dADDA (also dADDA$\rm{_{opt}}$). Figures~\ref{fig10}--\ref{fig100} 
	illustrates this influence, where $\alpha$ and $\beta$   take $200$ different values. 
	In all six figures,  the effects for  
	different $\alpha$ are displayed on the left while that for $\beta$ are on the right.  
	Besides the number of  iterations, we also plot the numerical results for ERres$_k$. 
	It shows in all left figures that  there are several ``critical" points in $\alpha$, 
	at which the number of iterations and ERres$_k$ jump  abruptly; 
	and for all right figures  it seems that one ``critical" point exists for $\beta$, 
	which grows as the size $n$ increases.

\begin{figure*}[ht]
		\centering
		\subfloat[$\beta =0.65574$]{
			\includegraphics[width=2.8in]{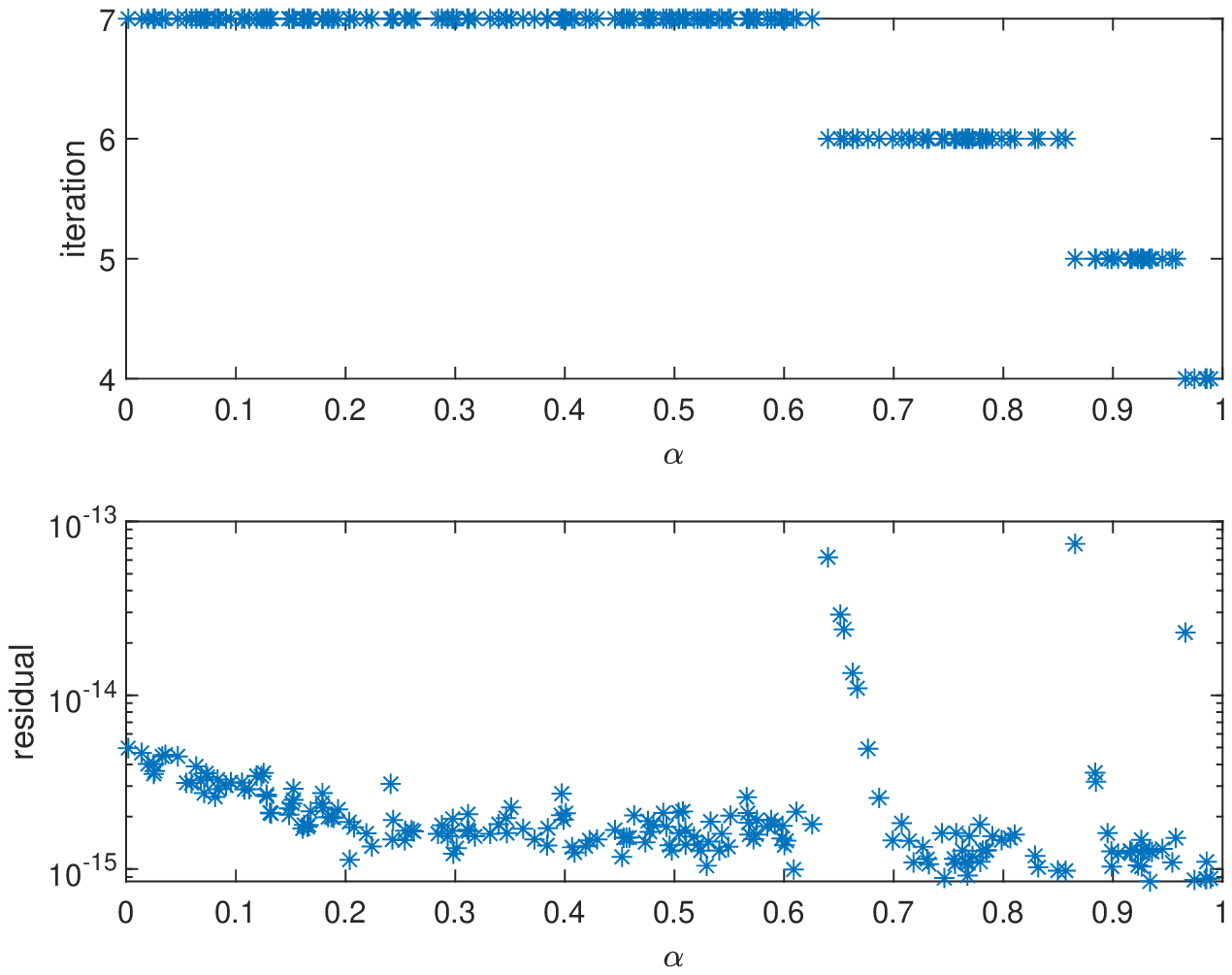}
		}
		\subfloat[$\alpha=0.4608$]{
			\includegraphics[width=2.8in]{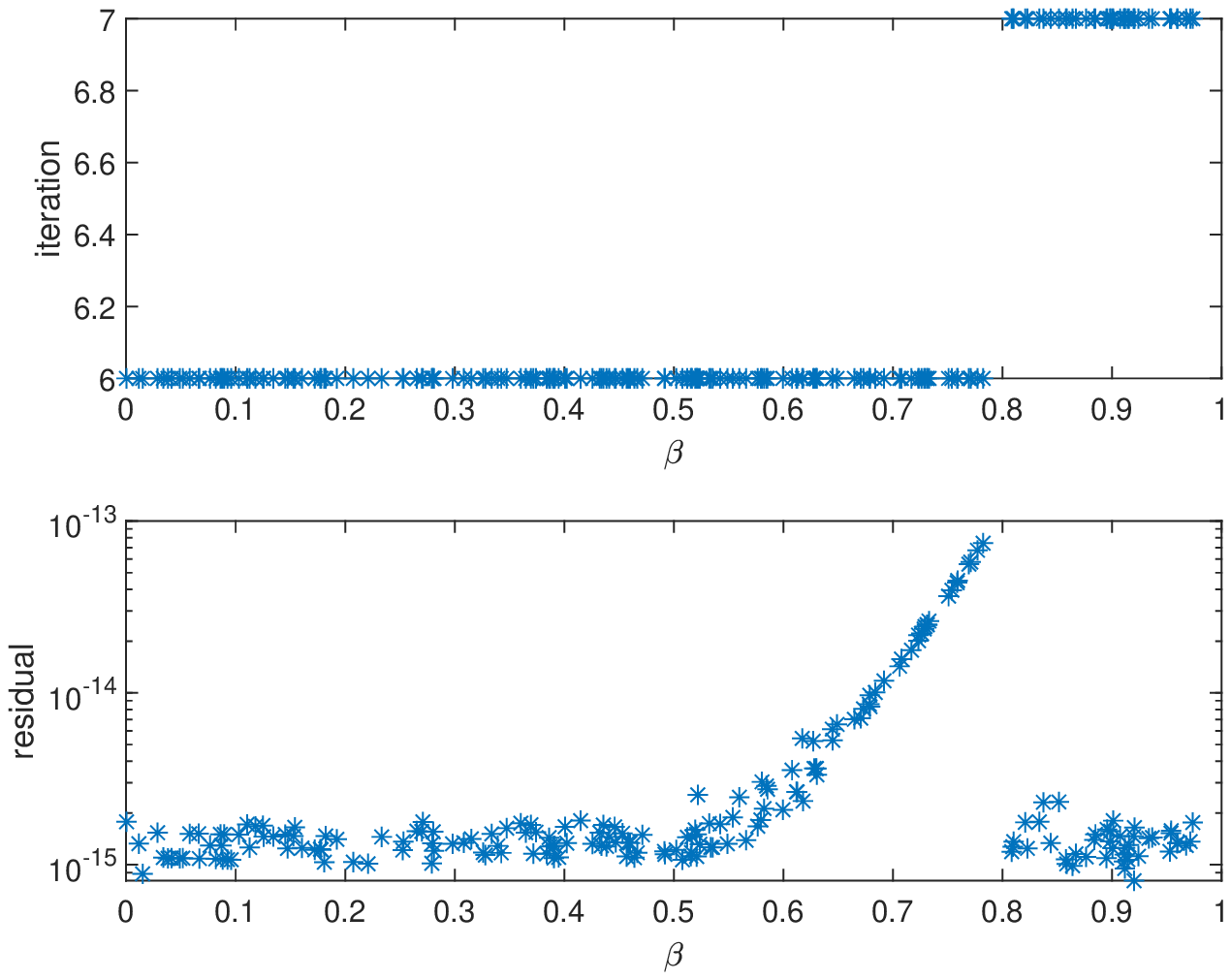}
		}
	\caption{Iterations and ERres  for $n=10$}
		\label{fig10}
	\end{figure*}

\begin{figure*}[ht]
		\centering
		\subfloat[$\beta=0.35068$]{
			\includegraphics[width=2.8in]{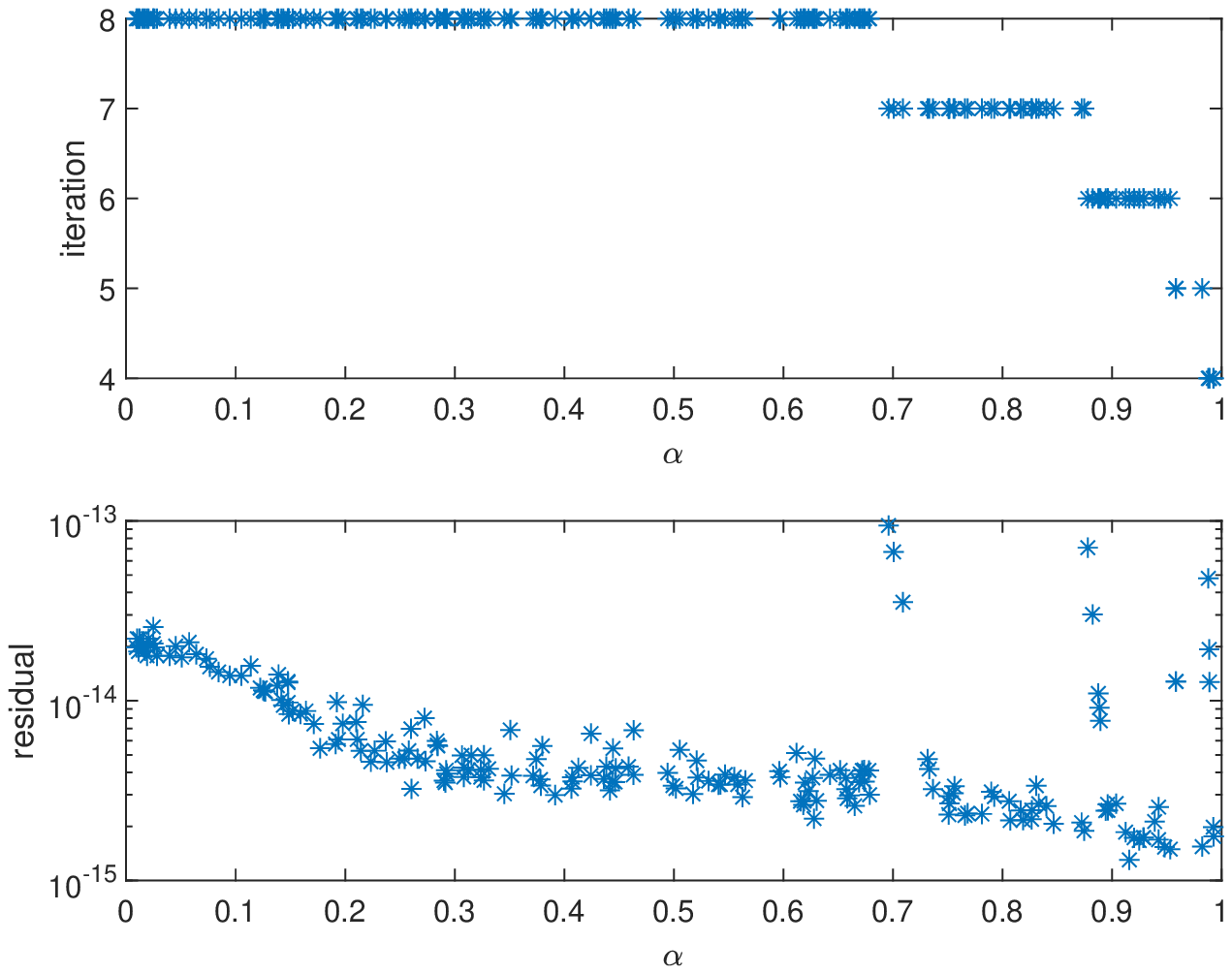}
		}
		\subfloat[$\alpha = 0.44585$]{
			\includegraphics[width=2.8in]{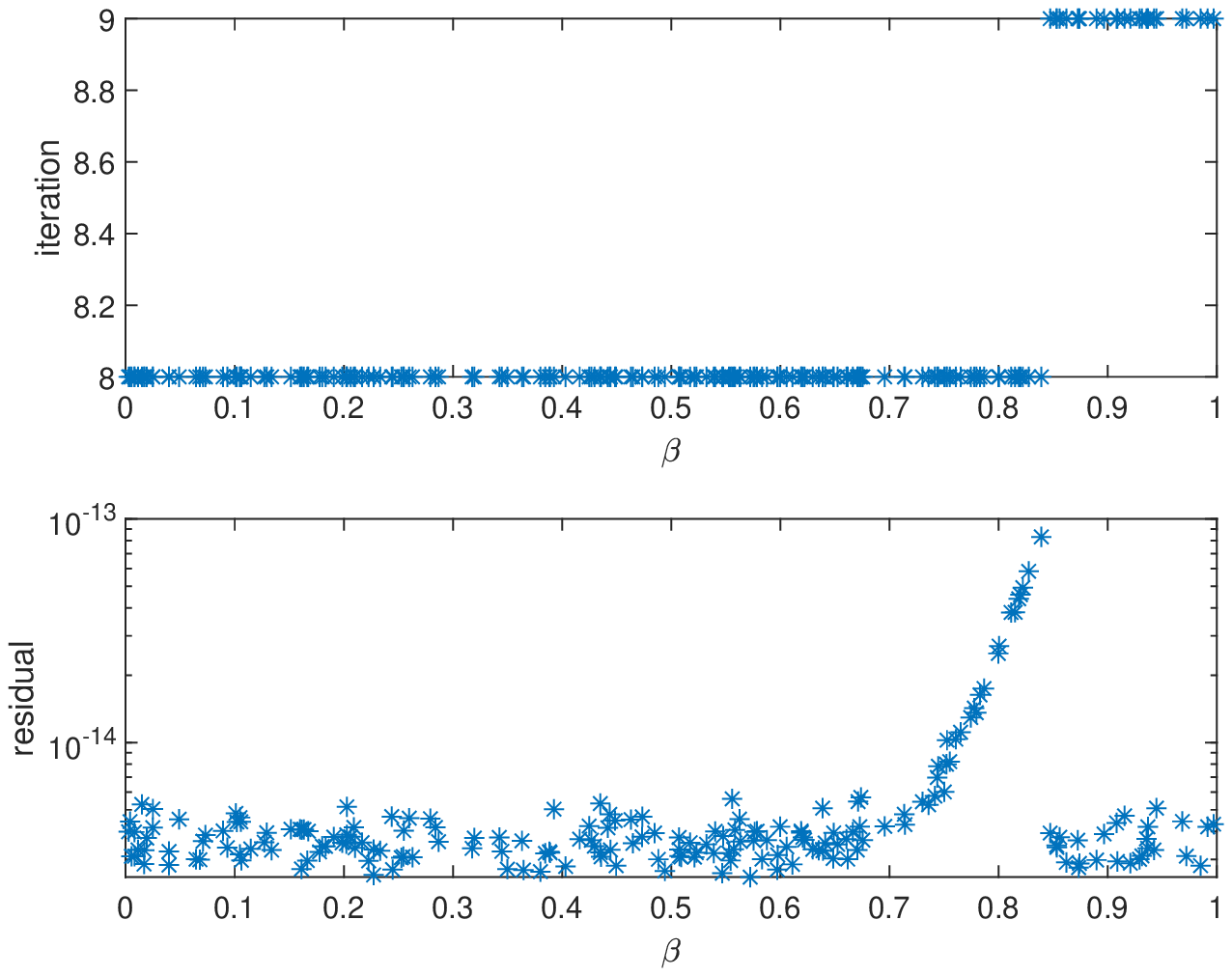}
		}
	\caption{Iterations and ERres  for $n=20$}
		\label{fig20}
	\end{figure*}

\begin{figure*}[ht]
		\centering
		\subfloat[$\beta=0.15662$]{
			\includegraphics[width=2.8in]{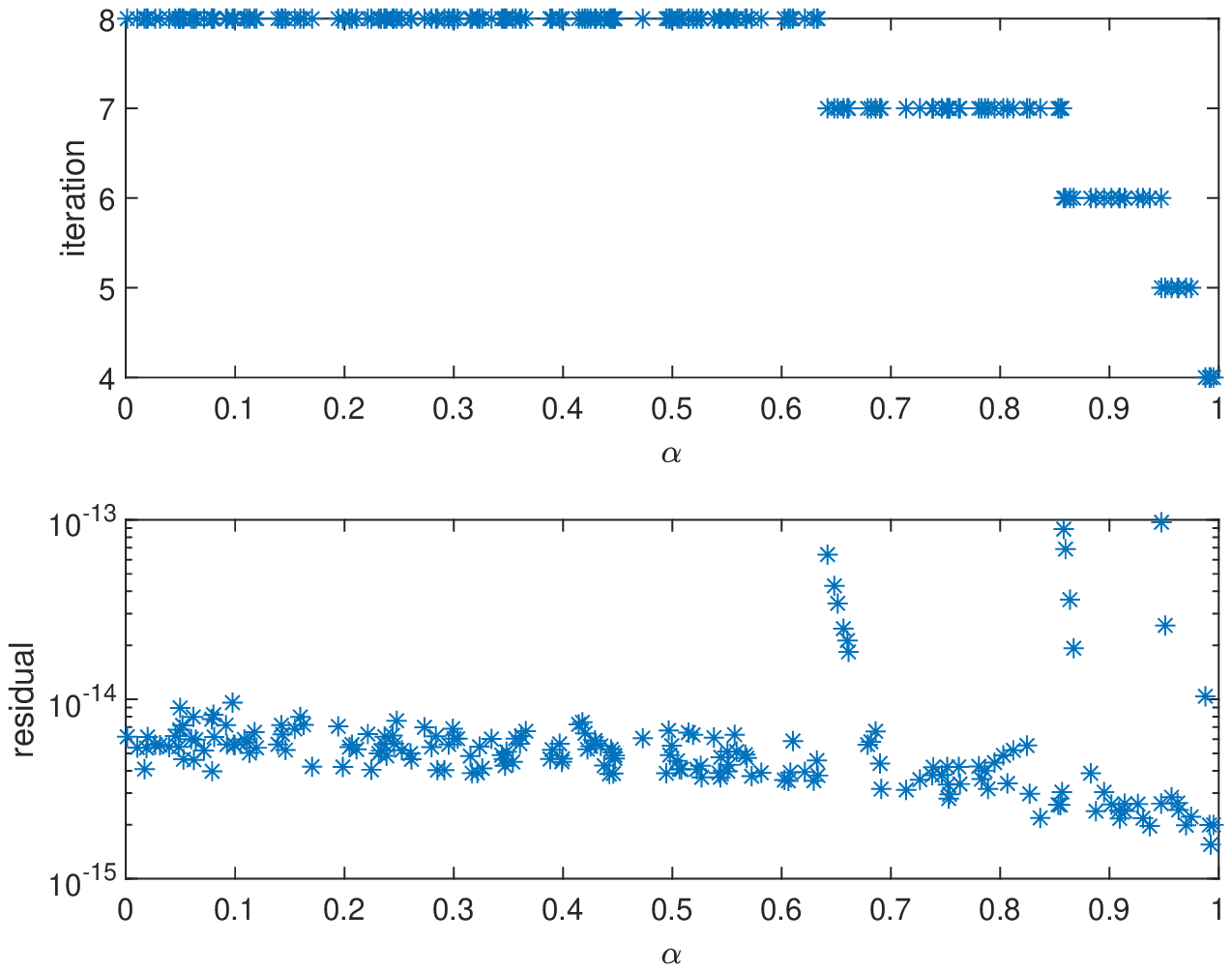}
		}
		\subfloat[$\alpha=0.71478$]{
			\includegraphics[width=2.8in]{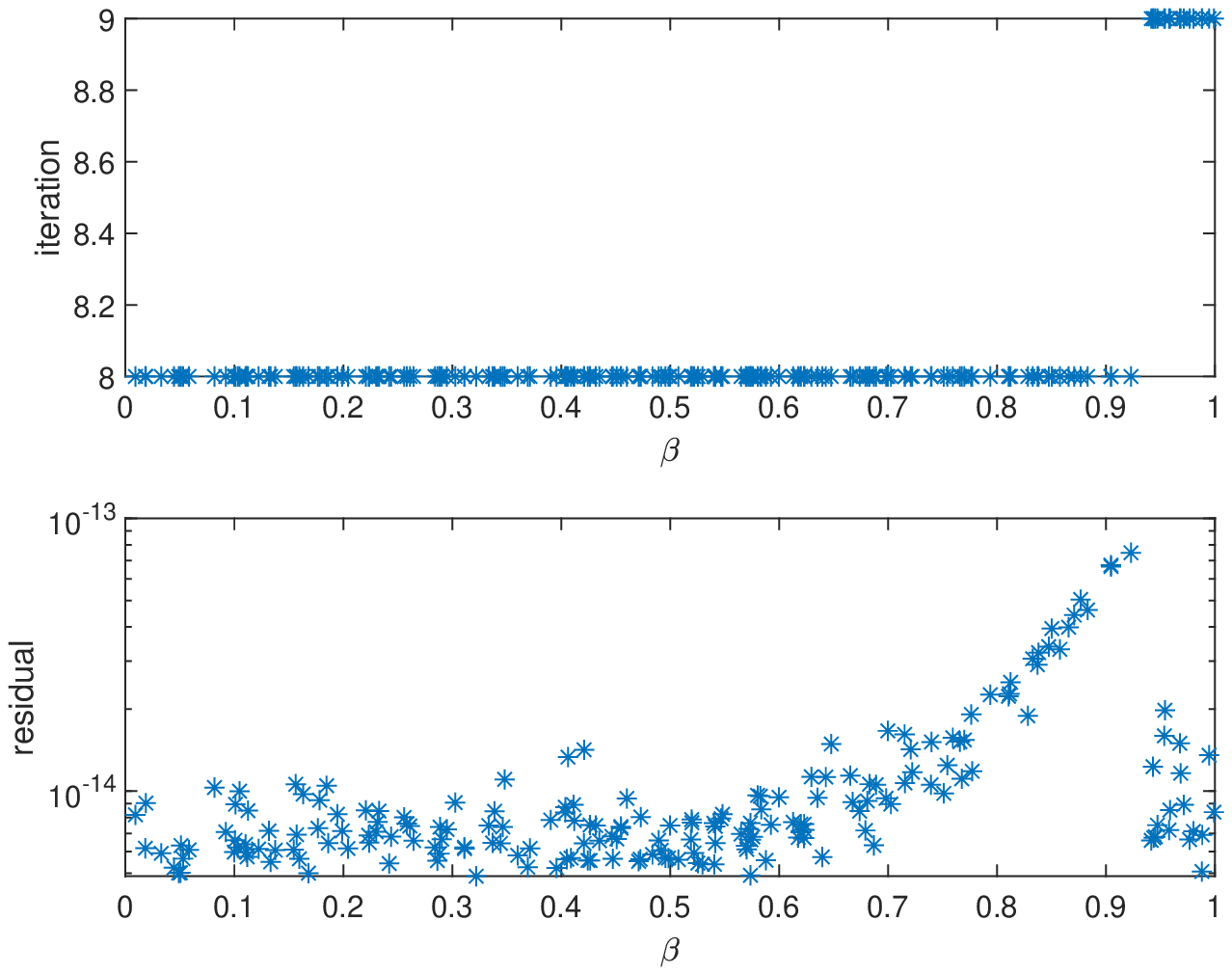}
		}
	\caption{Iterations and ERres  for $n=40$}
		\label{fig40}
	\end{figure*}

\begin{figure*}[ht]
		\centering
		\subfloat[$\beta=0.20647$]{
			\includegraphics[width=2.8in]{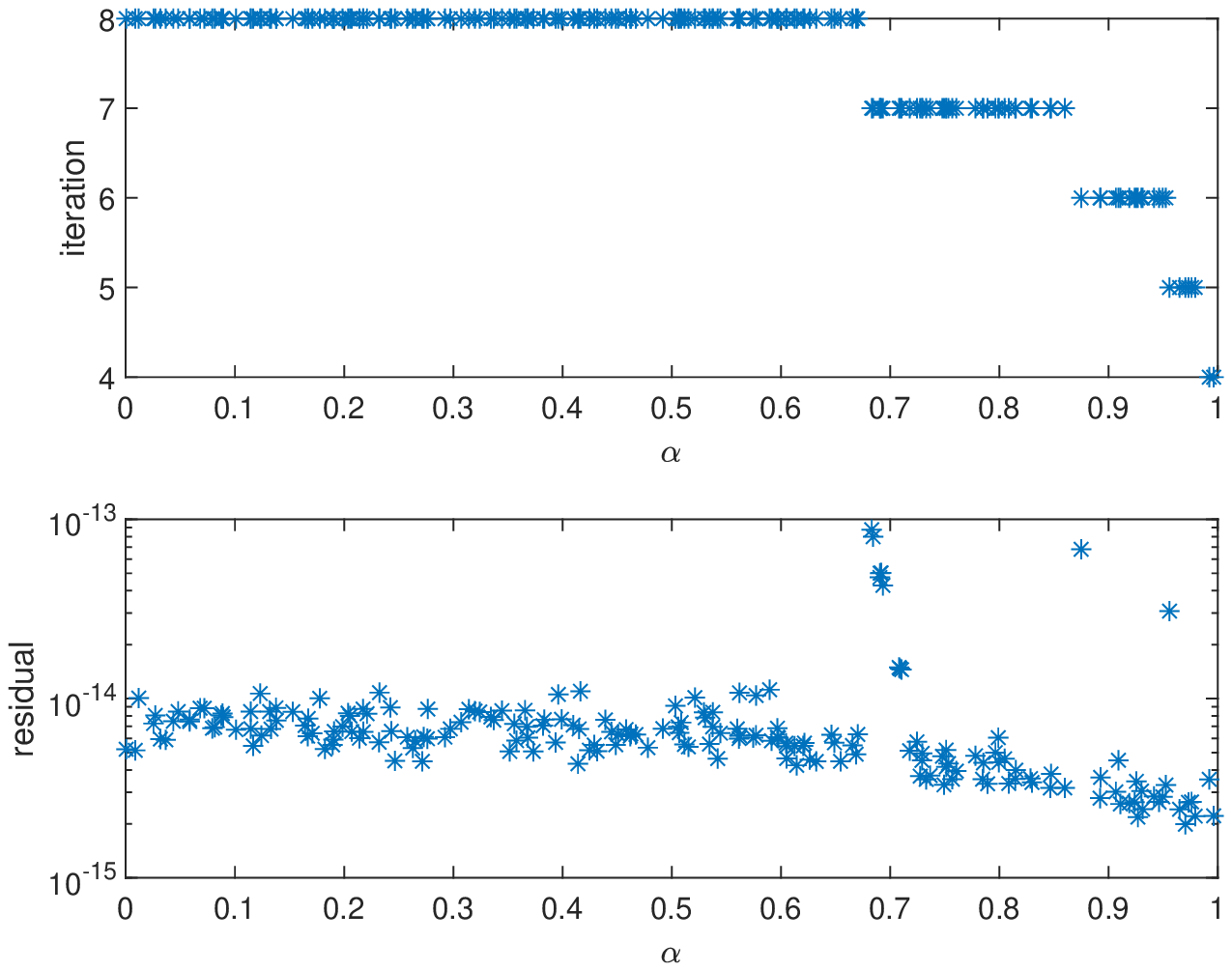}
		}
		\subfloat[$\alpha=0.7015$]{
			\includegraphics[width=2.8in]{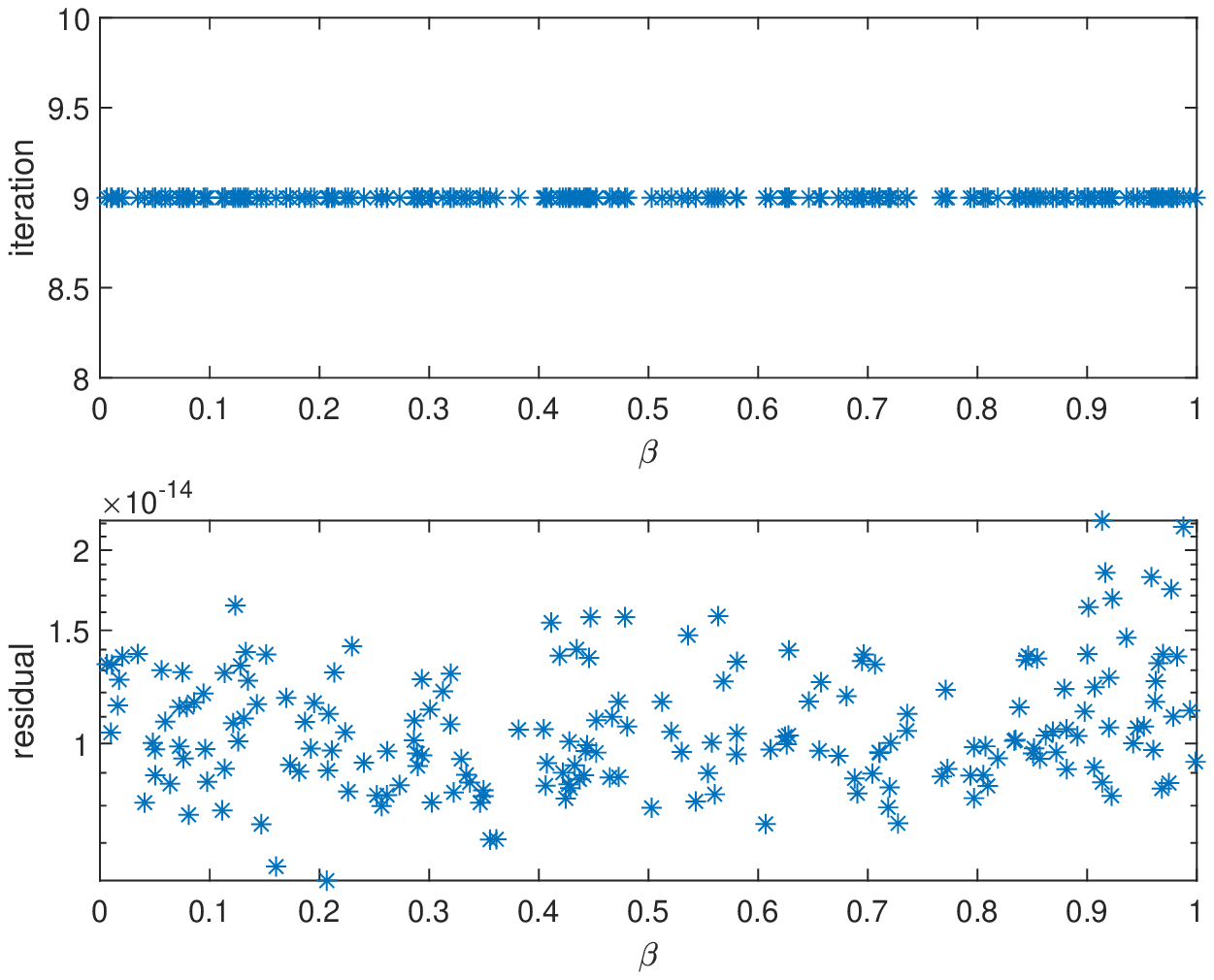}
		}
	\caption{Iterations and ERres  for $n=50$}
		\label{fig50}
	\end{figure*}

\begin{figure*}[ht]
		\centering
		\subfloat[$\beta=0.91508$]{
			\includegraphics[width=2.8in]{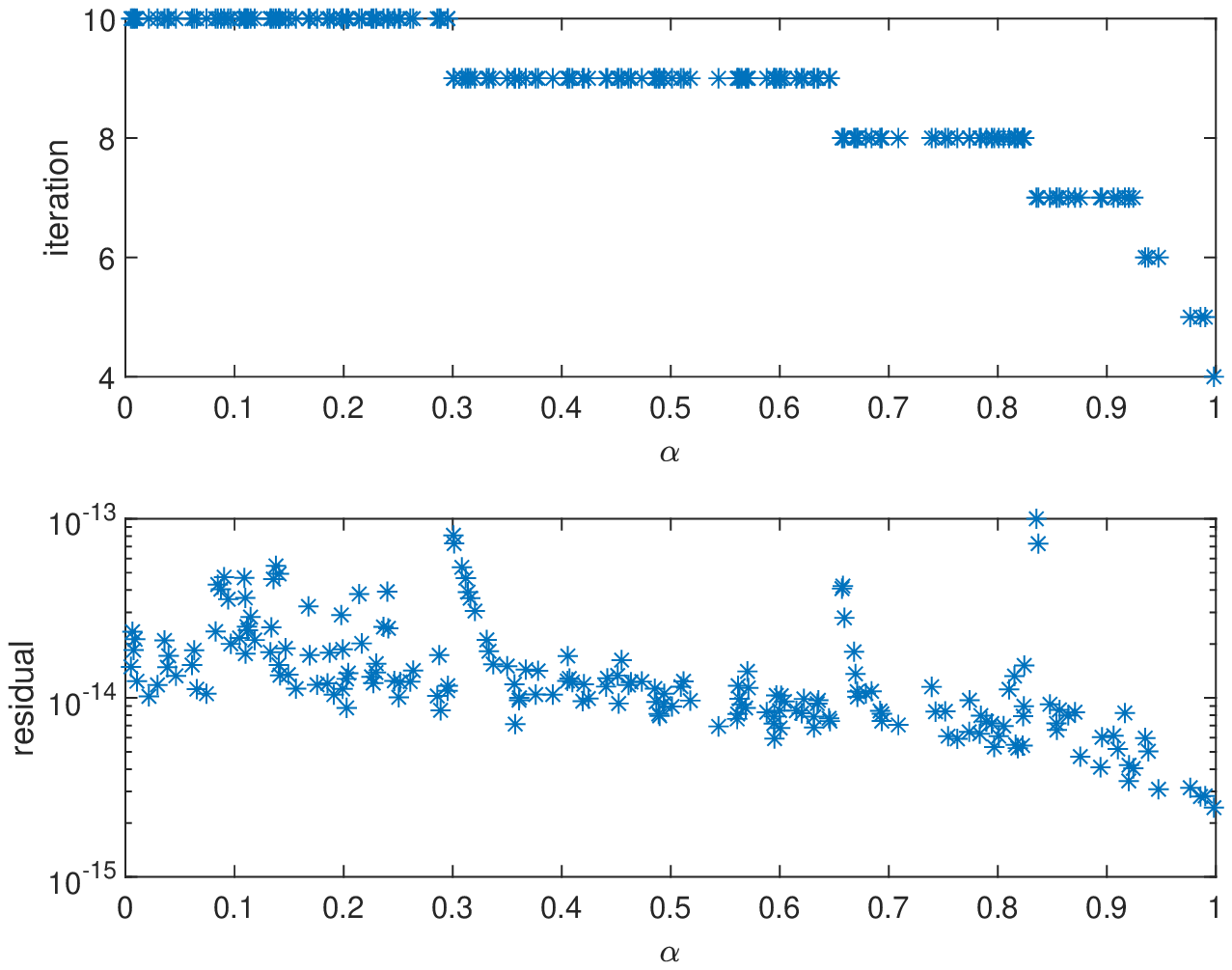}
		}
		\subfloat[$\alpha=0.78181$]{
			\includegraphics[width=2.8in]{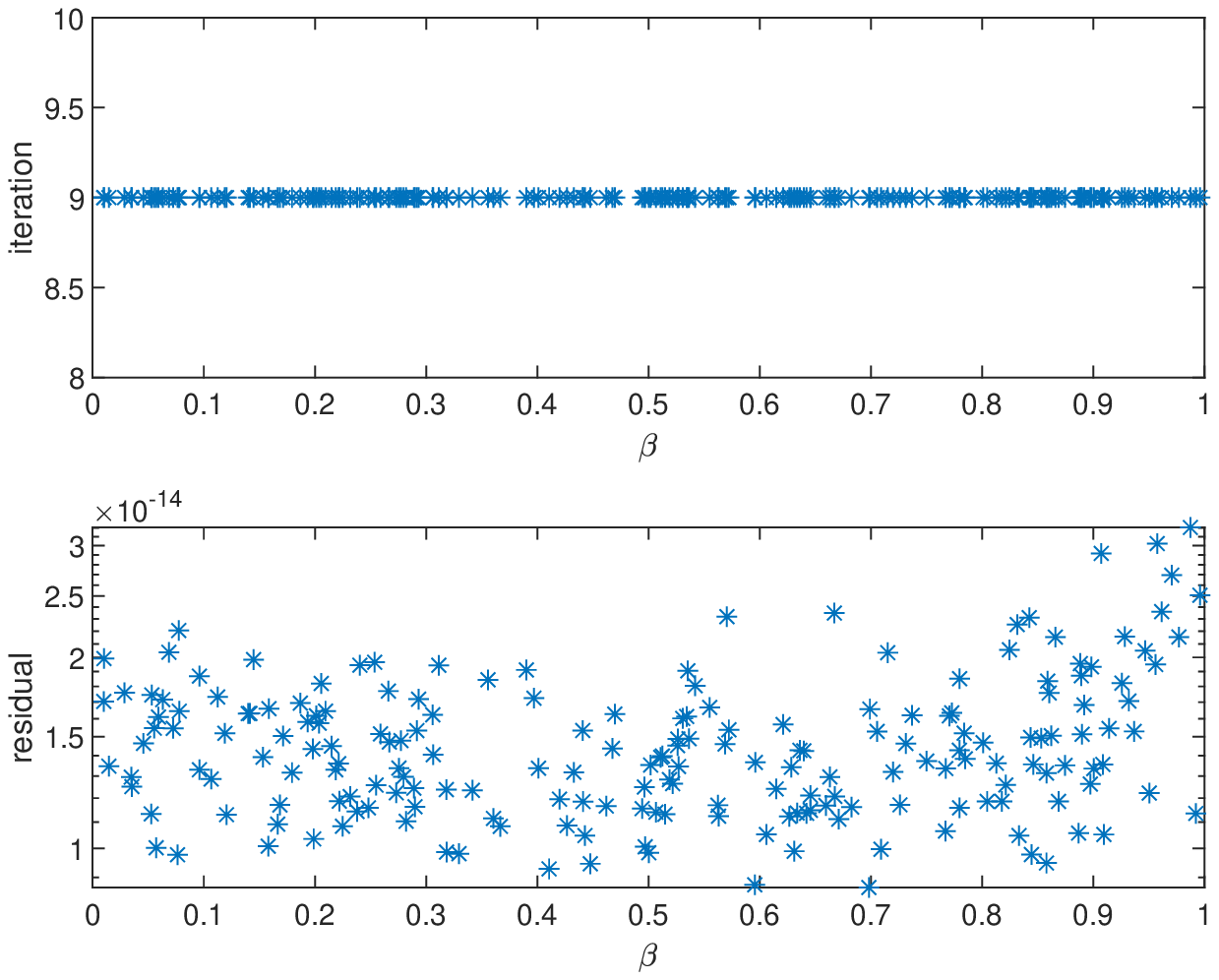}
		}
	\caption{Iterations and ERres  for $n=75$}
		\label{fig75}
	\end{figure*}

\begin{figure*}[ht]
		\centering
		\subfloat[$\beta=0.80531$]{
			\includegraphics[width=2.8in]{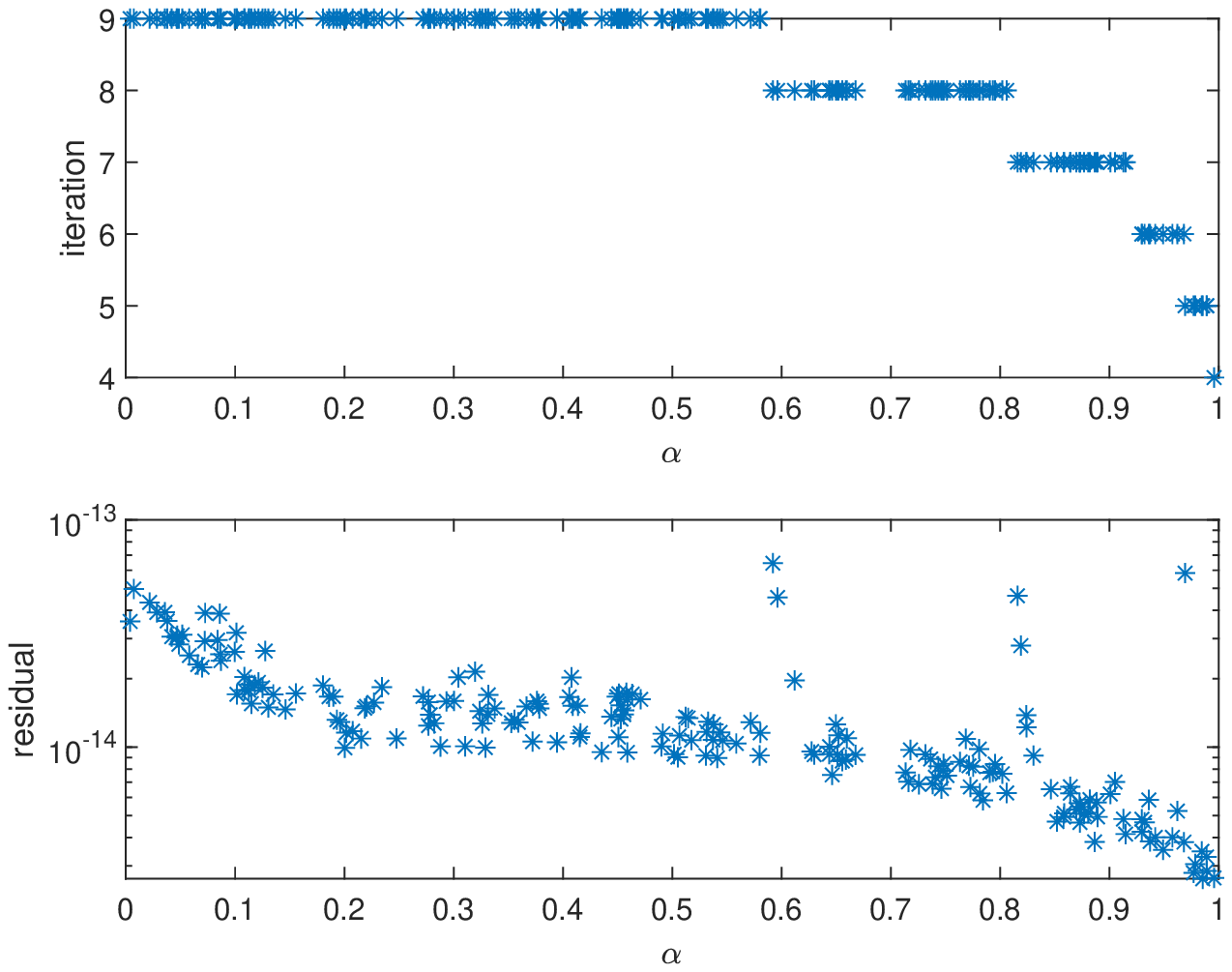}
		}
		\subfloat[$\alpha=0.79772$]{
			\includegraphics[width=2.8in]{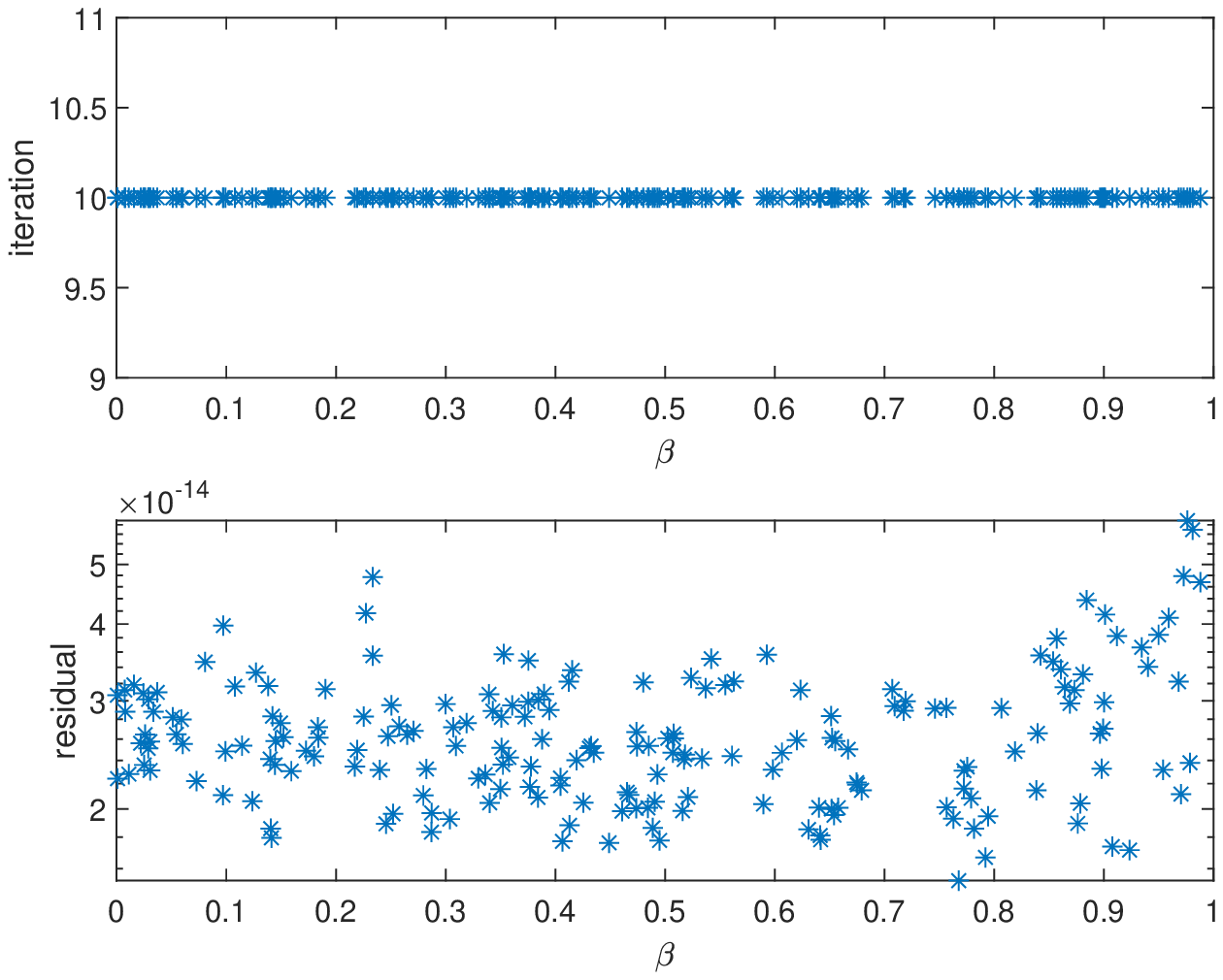}
		}
	\caption{Iterations and ERres  for $n=100$}
		\label{fig100}
	\end{figure*}

\end{example}

\section{Conclusions}\label{sec:conclusions}

The highly accurate alternating-directional doubling algorithm (accADDA) proposed by Xue and Li~\cite{xueL2017highly} is the most efficient method for computing the minimal nonnegative solution $X$ of MAREs of small sizes, 
which keeps the accuracy of all entries in $X$, especially the tiny ones. Illumined by the accDDDA, we propose a highly accurate algorithm for solving large-scale MAREs with low-rank structures. 
Firstly we show that the iteration recursions given in~\cite{xueL2017highly} can be decoupled, enabling  the highly accurate doubling algorithm  to solve large-scale MAREs. 
We prove the kernels in the decoupled form are M-matrices, and construct the novel triplet representations for these kernels, which is not a simple straightforward adaptation of that given in~\cite{xueL2017highly}. 
With these triplet representations, we develop the dADDA for large-scale MAREs with  low-rank structures. Associated linear equations are solved in a  cancellation-free manner, with the help of the GTH-like algorithm. 

Unlike the accADDA, our dADDA just applies  one iteration recursion for the solution $X$, making it possible to utilize the special structures that may exist in the original $A$ and $D$. 
Such as $A$ and $D$ are banded or low-rank  updates with diagonal matrices, 
 it only requires $\bigO(n+m)$ flops in each iteration. 
Numerical results illustrate the efficiency and superiority  of the proposed dADDA.


\bibliographystyle{siamplain}
\bibliography{narels.bib}

\end{document}